\pdfoutput=1

\documentclass[a4paper,11pt,reqno]{amsart}%

\usepackage{type1cm}  
\usepackage{graphicx} 
\usepackage{multicol} 
\usepackage[bottom]{footmisc}
\usepackage{amsfonts, amsmath, amssymb}
\usepackage{xspace} 
\usepackage{ytableau}
\usepackage{pgf}
\usepackage{subcaption}
\usepackage{multirow}
\usepackage{array}

\usepackage[draft=true]{minted} 
\setminted{breaklines=true, breakanywhere=true}
\definecolor{lbcolor}{rgb}{0.9,0.9,0.9}
\setminted{bgcolor=lbcolor}
\setminted{fontsize=\footnotesize}

\usepackage{mathtools}
\usepackage{booktabs}
\usepackage{tikz}     
\usetikzlibrary{calc} 
\usetikzlibrary{shapes.geometric}   
\usepackage{csquotes}        

\usepackage[colorinlistoftodos, bordercolor=orange, backgroundcolor=orange!20, linecolor=orange, textsize=scriptsize]{todonotes} 

\usepackage[noend]{algorithmic}

\usepackage{calc,ifthen,alltt,tabularx,capt-of}
\usepackage[linesnumbered,longend,ruled,vlined]{algorithm2e}

\usepackage[
   natbib,
   style=alphabetic,
   sorting=nty,
   maxbibnames=5, %
   maxalphanames=5, %
   isbn=false,url=false]
{biblatex}

\usepackage{lineno}
\usepackage[bookmarks=false, hypertexnames=false]{hyperref}
\hypersetup{
  colorlinks = true,          
  urlcolor = DeepSkyBlue3, linkcolor = Chartreuse4, citecolor = DarkOrange2
}

\numberwithin{equation}{section}

\usepackage{eqparbox,array}
\usepackage{dcolumn}
\usepackage{amsmath}
\usepackage{verbatim}%
\newcolumntype{r}{D{.}{.}{-1}}

\usepackage{longtable}
\usepackage{makecell}
\usepackage{wrapfig}

\usepackage{enumitem} %
\usepackage[matrix,arrow,curve]{xy}
\usepackage[capitalise, noabbrev]{cleveref} %

\newcommand\OSCAR{\textsc{OSCAR}\xspace}

\newcommand\Singular{\textsc{Singular}\xspace}

\newcounter{truefigure}
\makeatletter 
\renewcommand{\p@subfigure}{}
\makeatother

\hypersetup{
    colorlinks,
    linkcolor={red!70!black},
    citecolor={green!60!black},
    urlcolor={blue!80!black}
}

\newtheorem{theorem}{Theorem}[section]

\theoremstyle{plain}

\newtheorem{corollary}[theorem]{Corollary}

\theoremstyle{definition}
\newtheorem{definition}[theorem]{Definition}
\newtheorem{example}[theorem]{Example}

\theoremstyle{plain}

\newtheorem{proposition}[theorem]{Proposition}

\theoremstyle{remark}
\newtheorem{remark}[theorem]{Remark}
\numberwithin{equation}{section}

\usepackage[noend]{algorithmic}

\usepackage[linesnumbered,longend,ruled,vlined]{algorithm2e}

\usepackage[draft=true]{minted}
\setminted{breaklines=true}
\definecolor{lbcolor}{rgb}{0.9,0.9,0.9}
\setminted{bgcolor=lbcolor}
\setminted{fontsize=\footnotesize}

\usepackage{tabularx,float,capt-of}

\DeclareMathOperator{\Hom}{Hom}

\DeclareMathOperator{\mult}{mult}

\newcommand{\tmfloatcontents}{}
\newlength{\tmfloatwidth}
\newcommand{\tmfloat}[5]{
	\renewcommand{\tmfloatcontents}{#4}
	\setlength{\tmfloatwidth}{\widthof{\tmfloatcontents}+1in}
	\ifthenelse{\equal{#2}{small}}
	{\setlength{\tmfloatwidth}{0.45\linewidth}}
	{\setlength{\tmfloatwidth}{\linewidth}}
	\begin{minipage}[#1]{\tmfloatwidth}
		\begin{center}
			\tmfloatcontents
			\captionof{#3}{#5}
		\end{center}
\end{minipage}}

\newcommand {\PP}{{\mathbb P}}

\newcommand {\ZZ}{{\mathbb Z}}

\newcommand {\NN}{{\mathbb N}}
\newcommand {\TT}{{\mathbb T}}
\newcommand {\CC}{{\mathbb C}}

\newcommand\QQ{\mathbb{Q}} 
\DeclareMathOperator{\rad}{{rad}}
\DeclareMathOperator{\Lt}{{L}}
\def\lexF{\textrm{lex}}
\def\degrevlex{\textrm{degrevlex}}
\newcommand\II{{\rm I}}
\def\Cbar{{\overline C}}
\def\drlex{\textrm{drlex}}
\DeclareMathOperator{\Tor}{Tor}
\def\Input{\textbf{Input. }}
\def\Output{\textbf{Output. }}
\DeclareMathOperator{\subquo}{{subquo}}
\DeclareMathOperator{\imF}{im}
\def\pp{{\mathfrak p}}
\DeclareMathOperator{\Ext}{Ext}
\newcommand\cU{{\mathcal U}}
\newcommand\cM{{\mathcal M}}
\newcommand\cI{{\mathcal I}}
\newcommand\cF{{\mathcal F}}
\newcommand\cO{{\mathcal O}}
\newcommand\FF{\mathbb{F}}
\newcommand\GG{\mathbb{G}}
\DeclareMathOperator{\Spec}{Spec}
\def\sO{{\mathcal O}}

\newenvironment{tightarray}[1]{%
  \setlength{\arraycolsep}{0.5pt}%
  \begin{array}{#1}%
}{%
  \end{array}%
}

\begin{document}
\title{Commutative Algebra and Algebraic Geometry using \OSCAR}
\author{Janko Böhm, Wolfram Decker, Frank-Olaf Schreyer}
\address{Janko B\"ohm, Fachbereich Mathematik,
RPTU Kaiserslautern-Landau, Postfach 3049, 67653 Kaiserslautern, Germany}
\email{jboehm@rptu.de}
\address{Wolfram Decker, Fachbereich Mathematik,
RPTU Kaiserslautern-Landau, Postfach 3049, 67653 Kaiserslautern, Germany}
\email{wolfram.decker@rptu.de}
\address{Frank-Olaf Schreyer, Mathematik und Informatik,
Universität des Saarlandes, 66123 Saarbrücken, Germany}
\email{schreyer@math.uni-sb.de}

\keywords{Computational commutative algebra, computational  algebraic geometry, Gröbner bases, syzygies, free resolutions, cohomology, parametrizations, deformations, algebraic surfaces}

\subjclass[2010]{Primary 14Qxx, 13P10, 13D02; Secondary 14H50, 14Jxx, 14F05, 14Dxx, 14C30, 14D15}

\maketitle

\begin{abstract}
We give illustrative examples of how the computer algebra system \OSCAR can support research in commutative algebra and algebraic geometry. 
We start with a thorough introduction to Gröbner basis techniques, with particular emphasis
on the computation of syzygies, then apply these techniques to deal with ideal and ring theoretic concepts such as primary 
decomposition and normalization, and finally use them for geometric case studies which concern curves and surfaces, both 
from a local and global point of view.
\end{abstract}

\tableofcontents

\section{Introduction}

Most of mathematics is concerned at some level with setting up and solving various types of equations. Algebraic geometry 
is the mathematical discipline handling solution sets of systems of polynomial equations. By making use of a correspondence 
which relates the solution sets to ideals in polynomial rings, algebraic geometers translate problems concerning the geometry 
of the solution sets into algebra. Commutative algebra has developed to help solve these
problems as well as problems originating from number theory.

In modern times, support comes from computer algebra techniques, notably from the concept of Gröbner bases, where the main workhorse is Buchberger's algorithm for computing them. We will discuss a number of techniques relying on Buchberger's algorithm, with particular emphasis on the computation of syzygies. And we will show these techniques at work in \OSCAR, mainly focusing on geometrically motivated examples.

The integration of the commutative algebra and algebraic geometry functionality of \OSCAR with techniques of the other cornerstones allows for extended algorithmic applications. Examples include algorithms for computing GIT-fans with symmetries~\cite{gitbook}, fibration hopping for elliptically fibered K3 surfaces~\cite{BrandhorstZach2023}, algorithmic invariant theory~\cite{decker-schmitt-invariant-theory}, and F-Theory~\cite{chapter-F-theory-applications-in-oscar-book:2024}.

The textbooks \cite{GP08, cox2015} provide excellent introductions to both the geometry--algebra dictionary and the use of 
Gröbner bases in the study of geometric questions by algebraic means. We refer to them for basic details and proofs not given here.

\textbf{The Role of the Ground Field}  Throughout this chapter, $K$ will be an algebraically closed field, and $k\subset K$ a subfield.
Our geometric objects of study will be algebraic sets in affine or projective space over $K$ which are defined
by polynomial equations with coefficients in $k$. When it comes to explicit computations, we suppose that
the arithmetic of $k$ can be handled by \OSCAR. As a typical example, consider $K=\CC$ and $k=\QQ$.
In explicit examples, unless otherwise mentioned, we tacitly assume that we work with this pair.
Note that on the computational side, most questions considered can be checked using Gröbner bases.
A crucial point then is that Buchberger’s algorithm for finding such bases does not extend the given ground field. 
So most problems considered can be handled by computations over $k$. For example, based on Hilbert's Nullstellensatz, 
one can decide whether an algebraic system of equations defined over $k$ has a solution over $K$. On the other hand,
testing whether an ideal is radical cannot be decided using Gröbner bases alone. Note, however, that if $k$ is perfect, and 
$I\subset  k[x_{1},\ldots,x_{n}]$ is a radical ideal\footnote{If $R$ is a commutative ring, the \emph{radical} of an ideal 
$I \subset R$ is the ideal $\rad(I):=\{f\in R \mid \, \text{ there exists } N \ge 0  \hbox{ such that } f^{N}\in I \}.$ $I$ 
is a \emph{radical ideal} if $I=\rad(I)$.}, then also the extended ideal\footnote{The extended ideal $I K[x_{1},\ldots,x_{n}]$
is the ideal generated by the elements of $I$ in the larger polynomial ring obtained by extending the coefficient field from 
$k$ to $K$.} $I K[x_{1},\ldots,x_{n}]$ is a radical ideal (see \cite[VII, §11]{ZSCA}). We therefore tacitly assume that $k$ is perfect 
whenever the radical condition comes into play. The concept of primary decomposition depends on the choice of $k$. 
See Example \ref{ex:primdec}.

\section{Gr\"obner Basics and the Affine Geometry--Algebra Dictionary}

We write $\AA^{n} = \AA^{n}(K)$ for the \emph{affine $n$-space} over $K$.
Given a multivariate polynomial $f\in k[x_{1},\ldots,x_{n}]$, we call the set 
$$V(f) = \{ (a_{1},\ldots,a_{n}) \in K^{n} \mid f(a_{1},\ldots,a_{n})=0 \} \subset \AA^{n}$$
the \emph{vanishing locus} of $f$ (over $K$). The \emph{vanishing locus}
of finitely many polynomials $f_{1},\ldots,f_{r} \in k[x_{1},\ldots,x_{n}]$ 
is the simultaneous solution set
$$V(f_{1},\ldots,f_{r}) \;=\;\bigcap_{i=1}^{r} V(f_{i}) \subset \AA^{n}.$$
The \emph{ideal} generated by $f_{1},\ldots,f_{r}\in k[x_{1},\ldots,x_{n}]$ is written as
$$(f_{1},\ldots,f_{r}) :=\{ g_{1}f_{1}+\ldots+g_{r}f_{r} \mid g_{i} \in k[x_{1},\ldots,x_{r}] \}.$$
If $I$ is this ideal, we write $V(I) = V(f_{1},\ldots,f_{r})$.

\begin{theorem}[Hilbert's Nullstellensatz] With notation as above, $$V(f_{1},\ldots,f_{r})=\emptyset \;\Longleftrightarrow \; 1 \in (f_{1},\ldots,f_{r}).$$
\end{theorem}
\noindent
Thus, in our setting, the existence question for solutions over $K$ can be decided by solving the more general problem below:\\

\noindent
\emph{Ideal Membership Problem} Given $f,f_{1},\ldots,f_{r} \in k[x_{1},\ldots,x_{n}]$, decide whether $$f \in (f_{1},\ldots,f_{r}).$$

Taking our cue from the case of one variable, in which Euclidean division with remainder provides a solution to the ideal membership problem, we extend the division algorithm to polynomials in more than one variable, allowing at the same time more than one divisor. The extended algorithm, however, does not provide an immediate solution to the multivariate ideal membership problem.
Gröbner bases are designed to remedy this problem.

Univariate monomials are usually sorted according to their degree. The starting point for the concepts of multivariate division with remainder and Gröbner bases is that we have ways of sorting multivariate monomials.

\begin{definition} A \emph{monomial} in $k[x_{1},\ldots,x_{n}]$ is a power product
$$x^{\alpha}=x_{1}^{\alpha_{1}}\cdots x_{n}^{\alpha_{n}},  \;\text{ where }\; \alpha = (\alpha_{1}, \dots,  \alpha_{n})\in\NN^n.$$
A \emph{monomial ideal} in  $k[x_{1},\ldots,x_{n}]$ is an ideal generated by monomials.
A \emph{monomial ordering} $>$ on $k[x_{1},\ldots,x_{n}]$ is a total ordering $>$ on the set of monomials in $k[x_{1},\ldots,x_{n}]$ which is compatible with multiplication:
$$x^{\alpha} > x^{\beta} \implies x^{\alpha}x^{\gamma} > x^{\beta}x^{\gamma} \text{ for any triple } x^{\alpha}, x^{\beta}, x^{\gamma} \text{ of monomials}. $$
\end{definition}

A \emph{term} in $k[x_{1}\ldots,x_{n}]$ is the product $ax^{\alpha}$ of a scalar $a \in k$ and a monomial $x^{\alpha}$. If $0\neq f =\sum_{\alpha}f_{\alpha}x^{\alpha}\in 
k[x_{1}\ldots,x_{n}]$ is a polynomial  written as the finite sum of its nonzero terms, then the  \emph{leading term} of $f$ with respect to $>$ is
$$\Lt(f) \;\!:=\;\! \Lt_>(f) \;\!:=\;\! f_{\beta}x^{\beta}, \hbox{ where } x^{\beta}= \max \{ x^{\alpha} \mid f_{\alpha}\not=0 \}.$$

A monomial ordering $>$ is called \emph{global} if $x_{i} > 1$ for all $i$. In contrast, a \emph{local} monomial ordering satisfies $1 > x_{i}$ for all $i$. Note that $>$ is global iff it is a well-ordering.
The word global is used to indicate that Gröbner basis computations with respect to such orderings are used to study vanishing loci in their entirety.
Rather than taking this global point of view, however, one may be interested in examining the behaviour of such a set ``near''
 one of its points as we do in the next section.
Here, local orderings are used. \textbf{For the rest of this section}, $>$ will be a global monomial ordering on $k[x_{1},\ldots,x_{n}]$.
Unless otherwise indicated, all our considerations will be relative to this ordering.

\begin{theorem}[Division With Remainder]\label{thm:division} Let $f_{1},\ldots,f_{r} \in k[x_{1},\ldots,x_{n}]$ be nonzero. 
For each $f\in k[x_{1},\ldots,x_{n}]$, there is a unique expression
$
f=g_{1}f_{1}+\ldots +g_{r}f_{r} +h, {\text{ with }} g_{1},\ldots,g_{r}, h \in k[x_{1},\ldots,x_{n}],
$
and such that
\begin{enumerate}
\item\label{thm:divrem-1a} for $i > j$, no term of $g_{i}\Lt(f_{i})$ is divisible by $\Lt(f_{j})$;
\item\label{thm:divrem-2a} for all $i$, no term of $h$ is divisible by $\Lt(f_{i})$.
\end{enumerate} 
\end{theorem}

\begin{proof}[Existence] The statement is obvious if $f_{1},\ldots,f_{r}$ are terms. Applying this to $f$ and the $\Lt(f_{i})$, we get an expression
$$f= g_{1}^{(0)}\Lt(f_{1})+\ldots +g_{r}^{(0)}\Lt(f_{r})+h^{(0)},$$
where $h^{(0)}, g_{1}^{(0)},\ldots,g_{r}^{(0)}$ satisfy conditions \ref{thm:divrem-1a} and \ref{thm:divrem-2a} above.
Setting
$$f^{(1)}= f - (g_{1}^{(0)}f_{1}+\ldots +g_{r}^{(0)}f_{r}+h^{(0)}),$$
either $f^{(1)}$ is zero, and we are done, or $\Lt(f) > \Lt(f^{(1)})$. Iterating this, the resulting process must terminate 
since our underlying monomial ordering is assumed to be global and, thus, a well-ordering.
\end{proof}

The uniqueness part above corresponds to the fact that the division algorithm in 
the proof is \emph{determinate}\footnote{For many applications we may allow some indeterminacy 
of the division algorithm. See the \OSCAR user manual.} in that it does not make any choices in the process. We 
call $h$ the \emph{remainder upon determinate division of $f$ by $f_{1},\ldots,f_{r}$.} Note,
however, that due to condition \ref{thm:divrem-1a}, changing the order in which $f_{1},\ldots,f_{r}$ 
are listed may lead to a different remainder. Even worse, $f \in I = (f_{1},\ldots,f_{r})$ does not necessarily imply that $h$ 
is zero. As we will see, the latter is guaranteed if $f_{1},\ldots,f_{r}$ form a Gr\"obner basis of $(f_{1},\ldots,f_{r})$. 

\begin{definition} The \emph{leading ideal} of an ideal $I \subset k[x_{1},\ldots,x_{n}]$ is the monomial ideal
$$\Lt(I) \;\!:=\; \! \Lt_>(f) \;\!:=\; \! (\Lt(f) \mid f \in I).$$ 
A collection of elements $g_{1}, \ldots,g_{s}\in I$ is called a \emph{Gr\"obner basis} of $I$ if $$\Lt(I)=(\Lt(g_{1}),\ldots,\Lt(g_{s})).$$
\end{definition}
\noindent

\begin{proposition} Let $I \subset k[x_{1},\ldots,x_{n}]$ be an ideal.
\begin{enumerate}
\item 
If $g_{1},\ldots, g_{s}$ is a Gr\"obner basis of an ideal $I$, then a polynomial $f\in k[x_{1},\ldots,x_{n}]$
is contained in $I$ iff the remainder upon determinate division of $f$ by $g_{1},\ldots,g_{s}$ is zero.
\item The monomials $x^{\alpha}$ with $x^{\alpha} \notin \Lt(I)$ represent a $k$-vector space basis of the quotient $k[x_{1},\ldots,x_{n}]/I$.
We refer to these monomials as {\emph {standard monomials}} (for $I$, with respect to $>$).
\end{enumerate}
\end{proposition}

\begin{remark}\label{rem:normalform} If $g_{1},\ldots, g_{s}$ is a Gr\"obner basis of an ideal $I \subset k[x_{1},\ldots,x_{n}]$,
and $f\in I$, then the remainder upon determinate division of $f$ by $g_{1},\ldots,g_{s}$ is determined by $f$, $I$, 
and $>$ (and does not depend on the choice of Gr\"obner basis). It represents the residue class $f+I\in  k[x_{1},\ldots,x_{n}]/I $ 
in terms of the standard monomials. We speak of the \emph{normal form}  of $f$ mod $I$,  with respect to $>$.
\end{remark}

\noindent
If $I\subset k[x_{1},\ldots,x_{n}]$ is an ideal, a Gr\"obner basis of $I$ can be computed from any given set of generators using \emph{Buchberger's algorithm}.  This algorithm is a generalization of both Gaussian elimination for systems  of linear equations and Euclidean division with remainder in the univariate case. In Section \ref{syzygies}, we will make a few comments on how the algorithm works.

\begin{remark} A Gröbner basis $g_{1},\ldots, g_{s}$ computed by \OSCAR is \emph{minimal} in the sense that $\Lt(g_i)$ is not divisible by $\Lt(g_j)$, 
for $i \neq j$. A Gröbner basis is \emph{reduced} if it is minimal and no term of $g_i$ is divisible by $\Lt(g_j)$, for $i \neq j$. Note that these 
definitions of minimal and reduced deviate from those in textbooks as we do not ask that the leading coefficients of the 
Gröbner basis elements are one. A reduced Gröbner basis fulfilling this last property is unique.
\end{remark}

\begin{remark}
Gr\"obner bases were introduced by Gordan \cite{Gor99} who used them to give his own proof of Hilbert's basis theorem. In fact,
Gordan deduced the theorem for arbitrary ideals $I \subset k[x_{1},\ldots,x_{n}]$ from the combinatorial statement that monomial 
ideals are finitely generated (Dixon's Lemma).
\end{remark}

There is a multitude of monomial orderings. Of particular importance are the \emph{lexicographic ordering} $>_{\lexF}$ and the 
\emph{degree reverse lexicographic ordering} $>_{\degrevlex}$:
\begin{align*}
x^{\alpha}&>_{\lexF} x^{\beta}&:\Longleftrightarrow\;& \hbox{the first nonzero entry of $\alpha-\beta \in \ZZ^{n}$ is positive}.\\
x^{\alpha}&>_{\degrevlex} x^{\beta} &:\Longleftrightarrow\;& \deg x^{\alpha} >\deg x^{\beta} \hbox{ or }
( \deg x^{\alpha} =\deg x^{\beta} \hbox{ and } \\ 
&&&\hbox{ the last nonzero entry of $\alpha-\beta \in \ZZ^{n}$ is negative}).
\end{align*}
\noindent
Note that the above definitions depend on the ordering $x_1>\dots >x_n$ of the variables. The difference between $>_{\lexF}$ and $>_{\degrevlex}$
is subtle but crucial. For example, we have
$$x_{1}^{2} >_{\lexF} x_{1}x_{2} >_{\lexF} x_{1}x_{3} >_{\lexF} x_{2}^{2}$$
but
$$x_{1}^{2} >_{\degrevlex} x_{1}x_{2} >_{\degrevlex} x_{2}^{2}>_{\degrevlex} x_{1}x_{3}.$$
Using $>_{\degrevlex}$ aims at obtaining leading terms involving as few variables as possible. This typically has the effect that using Buchberger's algorithm with $>_{\degrevlex}$ results in a better performance concerning CPU time and memory usage.

The lexicographic ordering, on the other hand, can be used to eliminate any initial set of variables. In fact, the \emph{elimination property}
$$ \Lt(f) \in k[x_{i},\ldots,x_{n}] \;\Longrightarrow\; f \in k[x_{i},\ldots,x_{n}]$$ 
holds for each subring $k[x_{i},\ldots,x_{n}] \subset k[x_{1},\ldots,x_{n}]$. See Example \ref{ex:lex-elim} below.

\begin{remark}
So a single Gr\"obner basis computation with respect to $>_{\lexF}$
yields the whole flag of \emph{elimination ideals} $I_m=I \cap k[x_{m+1},\ldots,x_n]$,
$ m=0,\ldots,{n-1}$. If only  one of the elimination 
ideals is needed, other monomial orders are usually more efficient.
The \OSCAR function \mintinline{jl}{eliminate} takes this into account.
\end{remark}

\begin{example}\label{ex:lex-elim} Consider the ideal 
$$I=\left({\color{blue}x^{2}+y^{2}+2\,z^{2}-8},\,{\color{red}x^{2}-y^{2}-z^{2}+1},\,{\color{teal}x-y+z}\right) \subset \QQ[x,y,z].$$
Its vanishing locus $V(I) = \bigcap_{f \in I} V(f) \subset \AA^{3}(\CC)$ is the intersection of an ellipsoid, a hyperboloid, and a plane, 
see Figure~\ref{fig ellhypplane}.

\begin{figure}[ht]
\begin{center}
\includegraphics[scale=0.4]{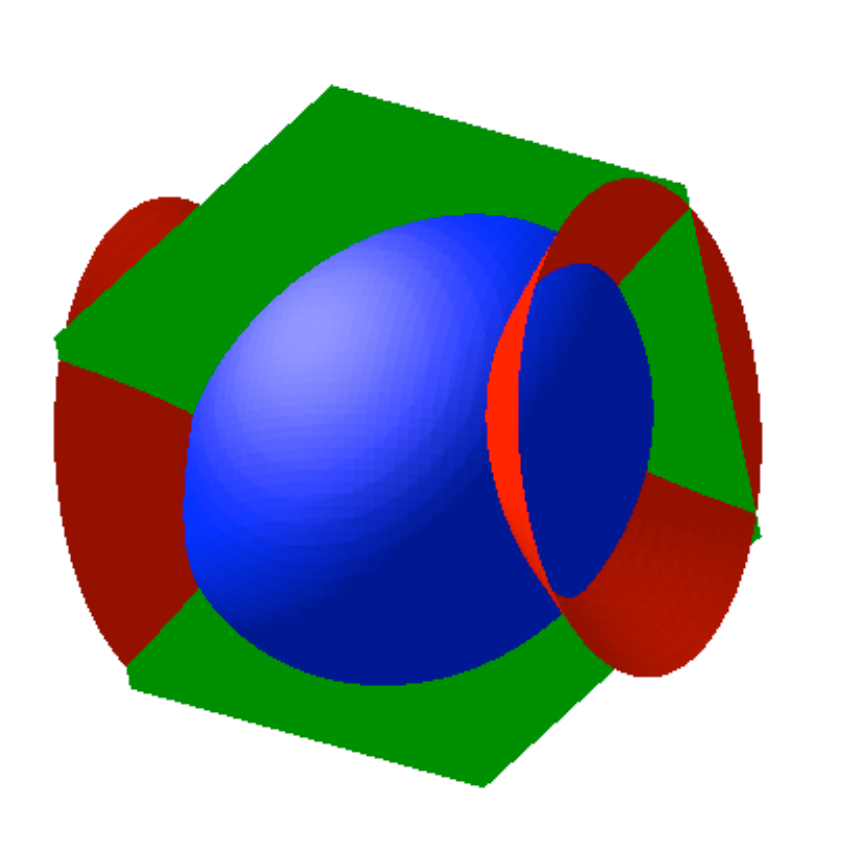}
\end{center}
\caption{Intersection of an ellipsoid, a hyperboloid, and a plane.}
\label{fig ellhypplane}
\end{figure}

\noindent

We use \OSCAR to compute a reduced lexicographic Gröbner basis of $I$:
\begin{minted}{jlcon}
julia> R, (x, y, z) = polynomial_ring(QQ, ["x", "y", "z"]);

julia> I = ideal(R, [x^2+y^2+2*z^2-8, x^2-y^2-z^2+1, x-y+z]);

julia> groebner_basis(I, ordering = lex(R), complete_reduction = true)
Gröbner basis with elements
1 -> 6*z^4 - 18*z^2 + 1
2 -> y + 3*z^3 - 9*z
3 -> x + 3*z^3 - 8*z
with respect to the ordering
lex([x, y, z])

\end{minted}  %
\noindent Solving the first equation $z^{4}-3\,z^{2}+\frac{1}{6}=0$, we find the four values $z=\pm \sqrt{\frac{3}{2}\pm\sqrt{\frac{9}{4}-\frac{1}{6}}}$. 
Substituting these into the other two equations, we get four points with coordinates $(-3z^{3}+8z,-3z^{3}+9z,z)$.
To determine the distance of each point to the origin, we use division with remainder:
\begin{minted}{jlcon}
julia> normal_form(x^2+y^2+z^2, I, ordering = lex(R))
-z^2 + 8
\end{minted}
\noindent Thus, the four points are grouped into two pairs of points with equal distance.

\begin{remark} \OSCAR functions such as \mintinline{jl}{groebner_basis} and  \mintinline{jl}{normal_form}
depend on the choice of a monomial ordering which is entered as a keyword argument as shown above. If no ordering 
is entered, the default ordering of the underlying ring $R$ is used. This ordering is \mintinline{jl}{degrevlex} except if 
$R$ is $\mathbb Z$-graded with positive weights. Then the corresponding \mintinline{jl}{wdegrevlex} ordering
is used:
\begin{minted}{jlcon}
julia> S, (x, y) = graded_polynomial_ring(QQ, ["x", "y"], [1, 2]);

julia> default_ordering(S)
wdegrevlex([x, y], [1, 2])

\end{minted} 
\end{remark} 

\begin{remark}  Each Gröbner basis computed in \OSCAR is cached and associated to the corresponding ideal for later use. 
In our example above, the cached lexicographic Gröbner basis of \mintinline{jl}{I} is used in the normal form computation. 
If subsequently another monomial ordering is specified when entering the command \mintinline{jl}{normal_form}, the required 
Gröbner basis will be computed and cached in addition to the lexicographic Gröbner basis.
\end{remark} 

\end{example}

\noindent
The triangular structure of the reduced lexicographic Gröbner basis in the example above is typical in the following sense: 

\begin{proposition}\label{prop elim} Let $I \subset k[x_{1},\ldots,x_{n}]$ be a radical ideal such that $V(I) \subset \AA^{n}$ is a finite set 
consisting of, say, $d$ points which project to different points on the $x_{n}$-axis. Then a reduced lexicographic Gr\"obner 
basis of $I$ has $n$ elements. One element is a polynomial in $k[x_{n}]$ of degree $d$, and for each $1\leq i \leq n-1$, there 
is an element contained in $k[x_{i},x_{n}]$ which depends linearly on $x_{i}$.
\end{proposition}

More generally, if $>$ is any global monomial ordering, then we have:

\begin{proposition} Let $I \subset k[x_{1}\ldots,x_{n}]$ be an ideal. Then $V(I) \subset \AA^{n}$ 
is finite iff only finitely many monomials $x^{\alpha}$ are not contained in $\Lt(I)$. In this case, the number 
of points in $V(I)$ is bounded by
$$|V(I)| \le \bigl|\{x^{\alpha} \mid x^{\alpha} \notin \Lt(I)\}\bigr|,$$
with quality holding if the extended ideal $I K[x_{1},\ldots,x_{n}]$  is the vanishing ideal of $V(I)$ as defined below.
\end{proposition}

\noindent
If $A\subset \AA^{n}$ is any subset, its \emph{vanishing ideal} is the ideal
$$\II(A) = \{ f \in K[x_{1},\ldots,x_{n}] \mid f(a)=0 \, \text{ for all } a \in A\}.$$
The residue ring $K[A]=K[x_{1},\ldots,x_{n}]/\II(A)$ can be identified with a $K$-subalgebra of the
$K$-algebra $K^{A} = \{ f \colon A \to K \}$ of all $K$-valued functions on $A$, namely the subalgebra which is generated by
the restrictions $x_{i}|_{A}$ of the coordinate functions $x_{i}$ to $A$. It is therefore called the \emph{coordinate ring} of $A$.

Note that vanishing ideals are radical ideals. 

\begin{theorem}[Hilbert's Strong Nullstellensatz] Let $I\subset K[x_{1},\ldots,x_{n}]$ be an ideal. Then
$$\II(V(I))=\rad(I).$$
\end{theorem}

A subset $A \subset  \AA^{n}$ is called an \emph{algebraic subset}, or simply an \emph{algebraic set}, if $A=V(I)$ for some ideal $I \subset K[x_{1},\ldots,x_{n}]$.
If $I$ is generated by polynomials $f_1,\dots, f_r \in k[x_{1},\ldots,x_{n}]$, then $A = V(f_1,\dots, f_r )$, and we call $k$ a \emph{field of definition} of $A$. 
Note that the algebraic subsets of $\AA^{n}$ form the closed sets of a topology on $\AA^{n}$. This topology is called the \emph{Zariski topology} on $\AA^{n}$.
An algebraic set $A$ is called \emph{irreducible} if
$$A=A_{1}\cup A_{2} \;\Longrightarrow\; A=A_{1} \hbox{ or } A=A_{2}$$
holds for every pair of nonempty algebraic subsets $A_{1},A_{2}$ of $\AA^{n}$. Note that $A$ is irreducible iff $\II(A)$ is a prime ideal
iff $K[A]$ is an integral domain. An \emph{affine variety} is an irreducible algebraic subset of some $\AA^{n}$. 

Every algebraic set $A$ is a finite union of irreducible algebraic sets $C_{i}$. If
$$A = C_{1}\cup \ldots \cup C_{r} \hbox{ with } C_{i} \not \subset C_{j}  \hbox{ for } i \not= j,$$
then each $C_{i}$ is called a \emph{component} of $A$.

\begin{example} \phantom{AAA}\\ 
\noindent
\begin{minipage}{6.5cm}
Consider the monomial ideal 
$$(xy,yz)\subset \QQ[x,y,z].$$
Its vanishing locus has two components:
$$V(xy,yz)=V(y)\cup V(x,z) \subset \AA^{3}.$$ 
Indeed, both $V(y)$ and $V(x,z)$ are irreducible since $K[x,y,z]/(y)\cong K[x,z]$ and $K[x,y,z]/(x,z)\cong K[y]$
are integral domains.
\end{minipage}  \hspace{1cm}
\begin{minipage}{4cm} \includegraphics[scale=0.18]{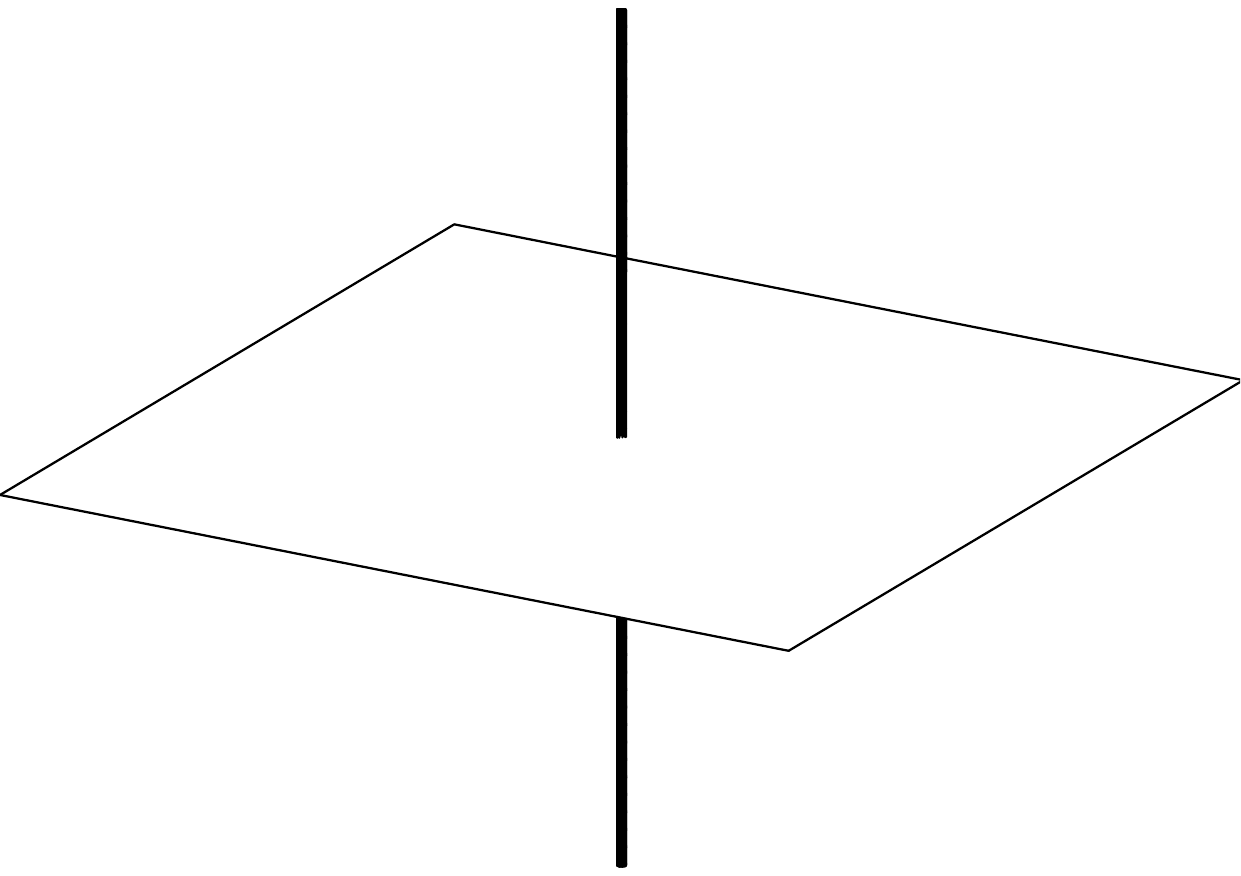}\end{minipage}
 \end{example} 

\begin{example}\label{ex:primdec}   On the algebraic side, decomposing an algebraic set is related to
the concept of primary decomposition. Here is \OSCAR code illustrating this:
\begin{minted}{jlcon}
julia> R, (x, y, z) = polynomial_ring(QQ, ["x", "y", "z"]);

julia> I = ideal(R, [x*y, y*z]);

julia> minimal_primes(I)
2-element Vector{MPolyIdeal{QQMPolyRingElem}}:
 Ideal (z, x)
 Ideal (y)

julia> J = ideal(R, [x^2+1])^2;

julia> primary_decomposition(J)
1-element Vector{Tuple{MPolyIdeal{QQMPolyRingElem}, MPolyIdeal{QQMPolyRingElem}}}:
 (Ideal (x^4 + 2*x^2 + 1), Ideal (x^2 + 1))

julia> L = absolute_primary_decomposition(J)
1-element Vector{Tuple{MPolyIdeal{QQMPolyRingElem}, MPolyIdeal{QQMPolyRingElem}, MPolyIdeal{AbstractAlgebra.Generic.MPoly{AbsSimpleNumFieldElem}}, Int64}}:
 (Ideal (x^4 + 2*x^2 + 1), Ideal (x^2 + 1), Ideal (x - _a), 2)

julia> base_ring(L[1][3])
Multivariate polynomial ring in 3 variables x, y, z
  over number field of degree 2 over QQ
\end{minted} 
\noindent
Note that there is a single primary component over $\QQ$. Furthermore, there is one corresponding class of conjugated absolute associated primes.
This is defined over a number field of degee two and is represented by \mintinline{jl}{L[1][3]}.
\end{example} 

\begin{theorem} 
The correspondences $V$ and $\II$ induce bijections
$$
\begin{matrix}
\{ \hbox{radical ideals of } K[x_{1},\ldots,x_{n}] \}  & \longleftrightarrow&\{\hbox{algebraic subsets of }\AA^{n} \}, \cr
 \cup & &\cup \cr
 \{ \hbox{prime ideals of } K[x_{1},\ldots,x_{n}] \} &  \longleftrightarrow&\{\hbox{irreducible alg. subsets of }\AA^{n} \},  \cr
 \cup && \cup \cr
 \{ \hbox{maximal ideals of } K[x_{1},\ldots,x_{n}] \}  & \longleftrightarrow& \{\hbox{points of }\AA^{n} \}.  \cr
 \end{matrix}
  $$
 \end{theorem}
 
If a system of algebraic equations  has infinitely many solutions, we might ask for the dimension of the solution set.
Equivalently, if $A\subset \AA^{n}$ is this set, we might ask for the Krull dimension of $K[A]$. Note that $A$ is finite 
iff its dimension is zero. We say that $A$ is \emph{equidimensional} of dimension $d$ if every component of $A$ has 
the same dimension $d$. If $A$ is equidimensional of dimension one (resp. two), we speak of a \emph{curve} (resp. 
\emph{surface}).
 
 \begin{theorem}[Gröbner Basis Criterion for Dimension]\label{GBDimensionCriterion}
 Let an ideal $I=(f_{1},\ldots,f_{r}) \subset k[x_{1},\ldots,x_{n}]$ be given.
 If $\rad(\Lt_>(I))=(x_{1},\ldots,x_{c})$ for some $c$, then $V(I)$ has dimension $n-c$ and the projection 
 $$\AA^{n} \to \AA^{n-c}, (a_{1},\ldots,a_{n}) \mapsto (a_{c+1},\ldots, a_{n}),$$
restricts to a finite  surjective map $V(I) \to \AA^{n-c}$.
If, moreover, $\Lt_>(I)$ is ge\-nerated by monomials in the subring $k[x_{1},\ldots,x_{c}]$, then $I$ is 
\emph{unmixed}\footnote{That is, all associated primes of $I$ have Krull dimension $n-c$.} of dimension $n-c$,
and the residue ring $k[x_{1},\ldots,x_{n}]/I$ is Cohen--Macaulay (see \cite{bruns1998cohen} for Cohen--Macaulay rings). In particular, 
$A$ is equidimensional of dimension $n-c$.
 \end{theorem}

\begin{example}\label{ex:dimTC} Consider the twisted cubic curve $C=V(x_{2}-x_{1}^{2},x_{3}-x_{1}x_{2}) \subset \AA^{3}$:
\begin{minted}{jlcon}
julia> R, (x1, x2, x3) = polynomial_ring(QQ, ["x1", "x2", "x3"]);

julia> I = ideal(R, [x2 - x1^2, x3 - x1*x2]);

julia> radical(leading_ideal(I))
Ideal generated by
  x2
  x1
\end{minted} 
\end{example}

\noindent
If $k$ is infinite, and $>$ is $>_{\degrevlex}$, then the assumption $$\rad(\Lt_>(I))=(x_{1},\ldots,x_{c}) \hbox{ for some $c$}$$ is satisfied  after a 
general linear change of coordinates. However, it is typically not a good idea to apply a change of coordinates as this may destroy any kind of 
sparseness of the given data. A more promising approach here is to use the equality $\dim V(I)= \dim V(\Lt(I))$ in order to reduce the computation 
of dimension to the purely combinatorial case of monomial ideals. The \OSCAR function \mintinline{jl}{dim} is based on this idea.
In Example \ref{ex:dimTC}:
\begin{minted}{jlcon}
julia> dim(I)
1
\end{minted} 
 
To study a given variety of dimension $\geq 1$, one may wish to parametrize the variety, for example by rational functions.
 
 \begin{example}\label{ex circ}
 The circle $x^{2}+y^{2}=1$ can be parametrized by rational functions. To find these functions, we intersect the circle ${\color{blue} x^{2}+y^{2}=1}$ 
with the one-parameter family of lines ${\color{red}y=tx+1} $ through $(0,1)$, obtaining one fixed and one variable intersection point. See Figure~\ref{fig circ}.

\begin{figure}[h]\center\includegraphics[scale=0.20,angle=0]{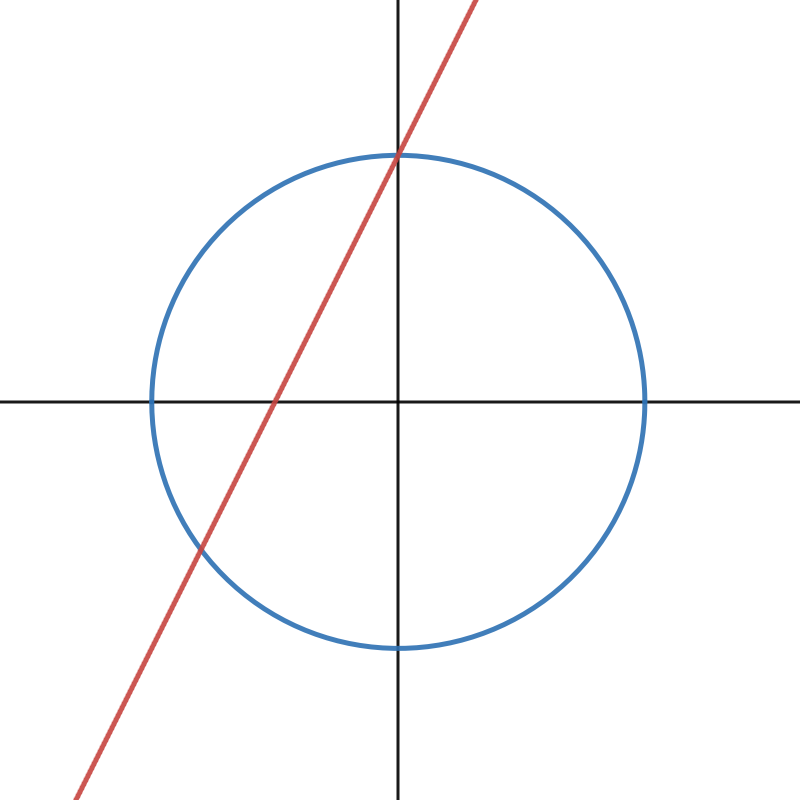}
  \caption{Parametrization of a circle.}
  \label{fig circ}
\end{figure}

\noindent
The following computation, which eliminates $y$, shows that either $x=0$ (the $x$-coordinate of the fixed point) or $x=\frac{-2t}{1+t^{2}}$ (the $x$-coordinate of the variable point).

\begin{minted}{jlcon}
julia> R, (x, y, t) = polynomial_ring(QQ, ["x", "y", "t"]);

julia> I = ideal(R, [x^2 + y^2 - 1, y - t*x - 1]);

julia> Gy = groebner_basis(I, ordering = lex([y, x, t]), complete_reduction=true)
Gröbner basis with elements
1 -> x^2*t^2 + x^2 + 2*x*t
2 -> y - x*t - 1
with respect to the ordering
lex([y, x, t])

julia> factor(Gy[1])
1 * (x*t^2 + x + 2*t) * x
\end{minted}
\noindent
Eliminating $x$, we observe that either $y=1$ or $y=\frac{1-t^{2}}{1+t^{2}}$.
\begin{minted}{jlcon}
julia> Gx = groebner_basis(I, ordering = lex([x, y, t]), complete_reduction=true)
Gröbner basis with elements
1 -> y^2*t^2 + y^2 - 2*y - t^2 + 1
2 -> x*t - y + 1
3 -> x*y - x + y^2*t - t
4 -> x^2 + y^2 - 1
with respect to the ordering
lex([x, y, t])

julia> factor(Gx[1])
1 * (y - 1) * (y*t^2 + y + t^2 - 1)
\end{minted}

\end{example}

\begin{remark}
In general, rational parametrizations do only rarely exist. The question of whether a variety of dimension 1 or 2 is rationally parametrizable is 
well understood, and for the affirmative case, we have algorithms to compute such a parametrization. For varieties of dimension $\ge 3$, 
this question is at the frontier of current research in algebraic geometry.
\end{remark}

The answers given in dimension 1  and 2 can be better  understood in the setting of projective algebraic geometry. Before we come to this,
we include a few words on local studies.

\section{Local studies}\label{sect:local-studies}

The local point of view will, for example, be taken when discussing the intersection multiplicity of two plane curves at a point. 
In general, for local studies, we may assume that the given point is the origin $o = (0,\dots, 0)$ of $\AA^n$. The case of an arbitrary 
point $p\in \AA^n$ can be dealt with by translating $p$ to $o$. This requires that we extend $k$ by adjoining each coordinate 
of $p$ not contained in $k$.

Algebraically, local studies lead us to enlarge $k[x_1, \dots, x_n]$ by considering the ring 
extension $k[x_1, \dots, x_n]\hookrightarrow \mathcal O_o$, where $\mathcal O_o$ is the localization of $k[x_1, \dots, x_n]$
at the maximal ideal corresponding to $o$. That is, $\mathcal O_o$ is the ring of fractions
$$
\mathcal O_o:=\left\{\frac{g}{h} \;\Bigg|\; g,h \in k[x_1, \dots, x_n], h(o)\neq 0\right\}.
$$
Here, $g/h$ stands for the equivalence class under the 
equivalence relation given by $(g,h)\sim(g', h')\Longleftrightarrow gh'=hg'$.

We now focus on plane curves.

\begin{definition}\label{def:int-mult}
Let $f, g\in k[x,y]$ be non-constant square-free polynomials without a common factor. Let $C, D\subset\AA^2$ be the plane 
curves defined by $f,g$, and let $p\in\AA^2$ be a point. Then the {\emph{intersection multiplicity}} of $C$ and $D$ at $p$ is 
defined as follows. If $p=o$, set
$$
i(C,D; o)= \dim_{k}(\mathcal O_o/ (f,g) \mathcal O_o).
$$
If $p$ is different from $o$, translate $p$ to $o$, extending $k$ if needed, and use the formula above over the extended field.
\end{definition}

\noindent
Intersection multiplicities can be found using Mora division with remainder 
to compute a Gröbner basis with respect to a local monomial ordering as implicitly done in the example below.
 
\begin{example}\label{ex:twocusps} We compute the intersection multiplicity of the tangential cusps 
$C=V(y^2-x^3)$ and $D=V(2y^2-x^3)$ at the origin: 
\begin{minted}{jlcon}
julia> R, (x, y) = polynomial_ring(QQ, ["x", "y"]);

julia> f = y^2-x^3;

julia> g = 2*y^2-x^3;

julia> U = complement_of_point_ideal(R, [0 ,0]);

julia> Rloc, _ = localization(R, U);

julia> I = ideal(Rloc, [f, g]);

julia> A, _ = quo(Rloc, I);

julia> vector_space_dimension(A)
6
\end{minted}
\noindent
Alternatively, we may use the built-in function \mintinline{jl}{intersection_multiplicity}:
\begin{minted}{jlcon}

julia> C = plane_curve(f)
Affine plane curve
  defined by 0 = x^3 - y^2

julia> D = plane_curve(g);

julia> P = D([0, 0])
Rational point
  of scheme(x^3 - 2*y^2)
with coordinates (0, 0)

julia> intersection_multiplicity(C, D, P)
6
\end{minted}
\end{example}

Our next definition introduces a local invariant which is 0 at every smooth point of a plane curve, and
measures the deviation of a singular point  from being a smooth point, otherwise. For this, we encourage
the reader to recall the concept of normalization of rings. Note that \OSCAR provides a number of
highly efficient algorithms for computing normalization. See \cite{GLS2010, BDLPSS2013, BDLP2022}. 
\begin{definition}
Let $f\in k[x,y]$ be non-constant and square-free, let $C = V(f)\subset \AA^2$, and let $p\in C$. The 
\emph{delta invariant} \( \delta_p(C) \) of $C$ at  \( p \) is defined as follows. If $p=o$, set
\[
\delta_o(C) = \dim_k \left( \overline{A_o} / A_o \right),
\]
where $A_o = \mathcal O_o/ f\mathcal O_o$ and $\overline{A_o}$ is the normalization of $A_o$. 
If $p$ is different from $o$, proceed as in Definition \ref{def:int-mult}.
\end{definition}

The \emph{total delta invariant} $\delta(C) = \sum_{p\in C }\delta_p(C)$ of an affine plane curve can be computed from 
the normalization of its coordinate ring as shown below:
\begin{example}\label{ex:quintic-curve}\phantom{AAA}
\begin{minted}{jlcon}
julia> R, (x, y) = polynomial_ring(QQ, ["x", "y"]);

julia> f = -3*x^5 - 2*x^4*y - 3*x^3*y^2 + x*y^4 + 3*y^5 + 6*x^4 + 7*x^3*y + 3*x^2*y^2 - 2*x*y^3 - 6*y^4 - 3*x^3 - 5*x^2*y + x*y^2 + 3*y^3;

julia> A, _ = quo(R, ideal(R, [f]));

julia> L = normalization_with_delta(A);

julia> L[3]
6
\end{minted}  %
\end{example}\label{ex:delta}
\noindent
We will come back to this example in the next section.

Different types of singularities of plane curves can often be distinguished by considering the tangent lines at these 
points. To define tangent lines to plane curves at singular points, note that each non-constant homogeneous  polynomial $f\in k[x,y]$
decomposes over $K$ into linear factors.

\begin{definition}
Let $C = V(f)\subset \AA^2$ be as above. Let
$
f=f_0 + f_1+f_2+\ldots+f_d
$
be the Taylor expansion of $f$ at $o$, where, for each $i$,  the polynomial $f_i$ collects the degree-$i$ terms of $f$.
The \emph{multiplicity $\mult(C,o)$ of $C$ at $o$} is defined to be the least $m$ such that $f_m\neq 0$. The 
\emph{tangent lines} to $C$ at $o$ are the lines defined by the linear factors of $f_{\mult(f,o)}$ over $K$.
We say that $o$ is an \emph{ordinary multiple point} of \( C \) if ${\mult(C,o)}\geq 2$ and all tangent lines at $o$
are distinct. For an arbitrary $p\in \AA^2$,  the definitions carry over as in Definition \ref{def:int-mult}.
\end{definition}

Note that the delta invariant and the multiplicity $m$ at an ordinary multiple point $p$ of $C$ are related as follows:
\( \delta_p(C) = \frac{m(m - 1)}{2} \). 

\section{Projective Algebraic Geometry}
Typically, two different lines in $\AA^{2}$ intersect in precisely one point. This is not true, however, for parallel lines. To treat pairs of lines on equal
footing, one has to add points at infinity, one for each class of parallel lines. This idea goes back to Renaissance artists who considered such points to introduce
perspective in their paintings. In any dimension $n$, the \emph{projective $n$-space} over $K$ is obtained as the union\footnote{In case $K=\CC$, 
the projective space $\PP^{n}$ carries the structure of a complex manifold which is a compactification of $\AA^{n}$.} 
$$\PP^{n}=\AA^{n}\cup H$$ 
of $\AA^{n}$ with a \emph{hyperplane at infinity} $H\cong \PP^{n-1}$.
To treat points of $\PP^{n}$ on equal footing, the formal definition of $\PP^{n}$ introduces \emph{homogeneous coordinates}. The \emph{homogeneous
coordinate ring} of $\PP^{n}$ is the graded polynomial ring $S=K[x_{0},\ldots,x_{n}]$, where each variable $x_{i}$ has degree $1$. 
\emph{Projective algebraic subsets} of $\PP^{n}$ are defined by homogeneous equations and correspond, thus, to homogeneous ideals $I\subset S$.
In this way, we get a projective geometry--algebra dictionary. In particular,
if $A\subset \PP^{n}$ is a projective algebraic set, we can speak of its
\emph{homogeneous vanishing ideal} $\II(A)\subset S$ and its \emph{homogeneous coordinate ring} $K[A]:=S/\II(A)$.
Moreover, the formal definition of $\PP^{n}$ gives rise to many ways of writing $\PP^{n}$ 
as the union of an “affine chart” $\AA^{n}$ and a hyperplane at infinity. Local concepts can, then, be extended from the affine to the projective case 
by considering a covering of $\PP^{n}$ by affine charts. For example,
$$
\PP^{n} = \bigcup_{i=0}^{n} U_i,
$$
with \emph{coordinate charts} $\AA^{n} \cong U_i:=\{((a_0:\dots : a_n) \in \PP^{n} \mid a_i\neq 0\}$ and hyperplanes at infinity $H_i =V(x_i)\cong \PP^{n-1}$.
A local question can then be studied by considering the question in an appropriate chart. For example,  the notion of intersection multiplicity 
of two plane curves at a point as introduced in the previous section carries over, independently of the choice of chart in which the point under consideration lies.
The \emph{projective closure} $\overline{A} \subset \PP^{n}$ of an algebraic set $A=V(I)\subset \AA^{n}\cong U_0$ is realized by homogenizing the elements of a 
Gröbner basis of $I$ with respect to a degree refining monomial ordering. Note that if $A$ is defined over $k$, then $\overline{A}$ has $k$ as a 
field of definition as well since Gröbner basis computations do not extend the ground field.
 
 \begin{example} Consider the twisted cubic curve $$C=V(x_{2}-x_{1}^{2},x_{3}-x_{1}x_{2}) \subset \AA^{3}.$$ Its projective closure 
$\Cbar$ is realized as follows:%
\begin{minted}{jlcon}
julia> R, (x1, x2, x3) = polynomial_ring(QQ, ["x1", "x2", "x3"]);

julia> I = ideal(R, [x2 - x1^2, x3 - x1*x2]);

julia> H = homogenizer(R, "x0"; pos=1);

julia> J = H(I)
Ideal generated by
  -x1*x3 + x2^2
  -x0*x3 + x1*x2
  -x0*x2 + x1^2
\end{minted}  %
 \noindent Note that it is not enough to homogenize the original two generators:
\begin{minted}{jlcon}
julia> J1 = ideal([H(I[1]), H(I[2])]) 
Ideal generated by
  x0*x2 - x1^2
  x0*x3 - x1*x2

julia> MP = minimal_primes(J1)
2-element Vector{MPolyIdeal{MPolyDecRingElem{QQFieldElem, QQMPolyRingElem}}}:
 Ideal (x1*x3 - x2^2, x0*x3 - x1*x2, x0*x2 - x1^2)
 Ideal (x1, x0)

julia> J in MP
true
\end{minted}  %
 \noindent %
 This shows that $V(J_1)\subset \PP^2$ contains the additional line $L=V(x_{0},x_{1})$ which is completely contained in the hyperplane  $H=V(x_{0})$ at infinity. Moreover:
\begin{minted}{jlcon}
julia> pt = radical(J+ideal([base_ring(J)[1]]))
Ideal generated by
  x2
  x1
  x0
\end{minted}  %
 \noindent Thus, the projective closure of $C$ intersects $H$ in the single point
 $$\Cbar \cap H =\{(0:0:0:1)\}$$ 
 which is the limit of the points $(1:t:t^{2}:t^{3})=(\frac{1}{t^{3}}:\frac{1}{t^{2}}:\frac{1}{t}:1)$ for $t\to \infty$.
 \end{example}
  
In $\PP^2$, we have the line at infinity. Any other line in $\PP^2$  is a line $L\subset \AA^2$  together with the 
common point at infinity of all lines parallel to L. Now, two different lines meet in precisely one point. In fact, 
much more is true:

\begin{theorem}[B\'ezout] Let $C$ and $D$ be projective plane curves of degrees $c$ and $d$, respectively. Suppose that the 
two curves do not have a common component. Then, counted with multiplicity, $C$ and $D$ intersect in precisely $c\cdot d$ points:
 $$ \sum_{p \in C\cap D} i(C,D;p) = c\cdot d.$$
 \end{theorem}
 \noindent
Here, the {\emph{degree}} of a projective plane curve is the degree of a square-free homogeneous polynomial defining it.

 \begin{example} Consider the projective closures $C$ and $D$ of the affine plane curves defined by the polynomials
$\color{blue}{f=(x-y)((x+y)^2-(x-y)^3)-(x+y)^4 }$ and $\color{red}{g=y^{2}-x^{2}+3x^{3} }$, respectively:
 \begin{center}
 \includegraphics[scale=0.2,angle=-90]{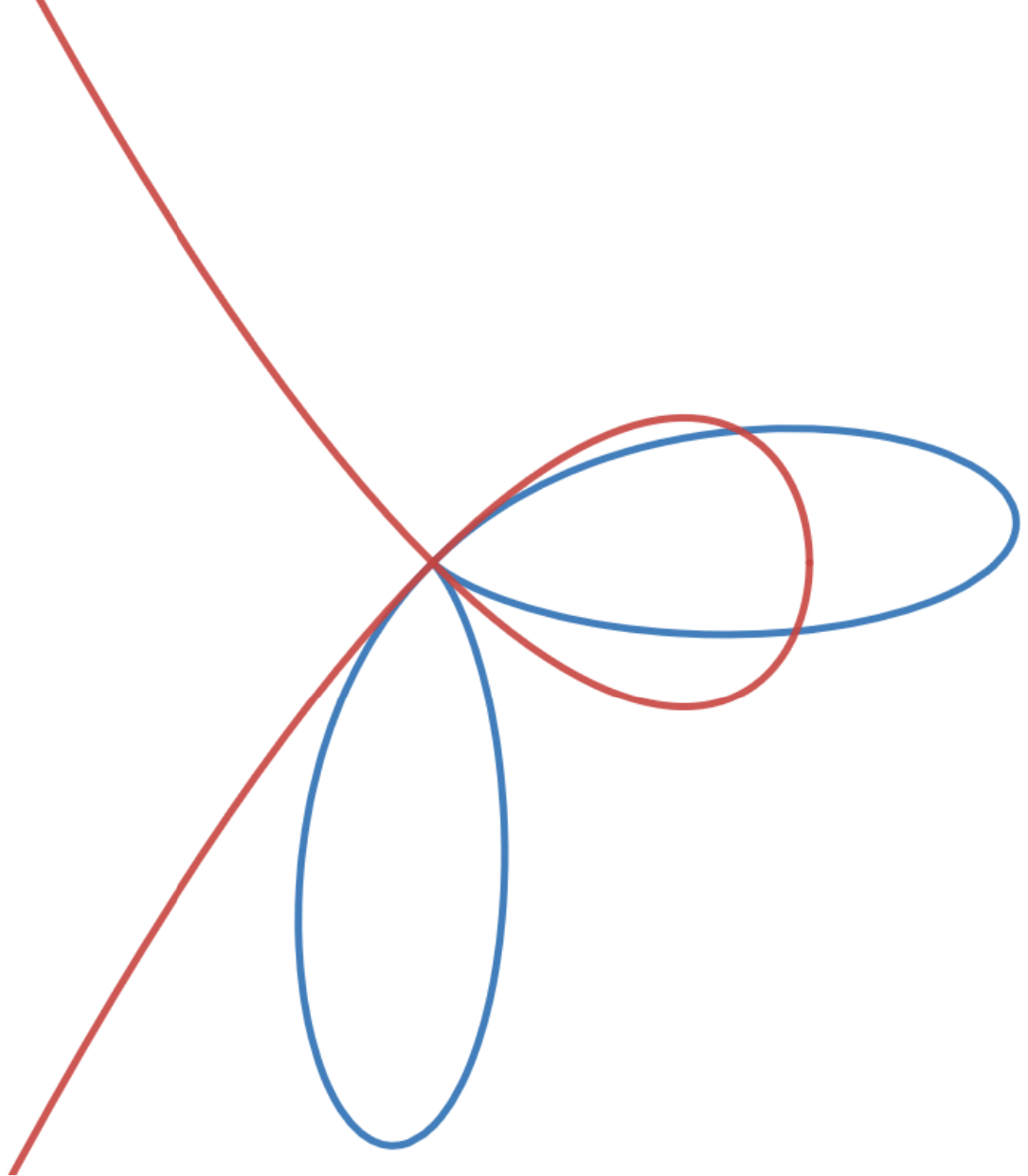}
 \end{center}
We show that $i(C,D;p) = 8$. In this example, as an alternative to the recipe from Example \ref{ex:twocusps}, 
we can also proceed as follows:
\begin{minted}{jlcon}
julia> R, (x, y) = polynomial_ring(QQ, ["x", "y"]);

julia> f = (x-y)*((x+y)^2-(x-y)^3)-(x+y)^4;

julia> g = y^2-x^2+3*x^3;

julia> H = homogenizer(R, "z");

julia> F = H(f);

julia> G = H(g);

julia> I = ideal([F, G]);

julia> degree(radical(I))
5

julia> cI = primary_decomposition(I);

julia> degrees = map(comp -> degree(comp[1]), cI)
2-element Vector{ZZRingElem}:
 4
 8
\end{minted}
\noindent
The \OSCAR implementation of the primary decomposition algorithm applied above depends on some random choices. 
As a result, another run of the algorithm may return the primary components in a different order.
\begin{minted}{jlcon}
julia> pos = findfirst(deg -> deg == 8, degrees)
2

julia> cI[pos][2]
Ideal generated by
  y
  x
\end{minted}
\noindent
 There is a conjugated group of four more intersection points, two of which are visible and real, as one can see from the minimal polynomial of the component in the absolute primary decomposition
 of $I=(f,g) \subset \QQ[x,y]$: 
\begin{minted}{jlcon}
julia> cIabs = absolute_primary_decomposition(I);

julia> pos_abs = findfirst(cp -> degree(cp[1]) == 8, cIabs)
2

julia> degree(cIabs[pos_abs][1])
8

julia> other_pos_abs = pos_abs == 1 ? 2 : 1
1

julia> cI2 = cIabs[other_pos_abs][3]
Ideal generated by
  648*y + (-160*_a^3 - 1269*_a^2 + 22446*_a + 972)*z
  184147758075888*x + (2850969000960*_a^3 + 22611747888864*_a^2 - 399955313722176*_a - 17319636680832)*y + (2884707374400*_a^3 + 22782606070410*_a^2 - 405172045313820*_a - 57843867366864)*z

julia> R2 = base_ring(cI2)
Multivariate polynomial ring in 3 variables over number field graded by 
  x -> [1]
  y -> [1]
  z -> [1]

julia> K = coefficient_ring(R2)
Number field with defining polynomial 160*x^4 + 1269*x^3 - 22446*x^2 - 972*x - 648
  over rational field

julia> a = gen(K);

julia> mp =  minpoly(a)
x^4 + 1269//160*x^3 - 11223//80*x^2 - 243//40*x - 81//20

julia> S, (x,) = polynomial_ring(QQ, ["x"]);

julia> g = mp(x);

julia> sols, _ = real_solutions(ideal([g]));

julia> length(sols)
2
\end{minted}
\noindent
Since $8+4=4\cdot 3$, there are no intersection points on the line at infinity.
 \end{example}

Most curves do not have a rational parametrization. In the case $K=\CC$, there is a topological obstruction. Let $C \subset \PP^{n}$ be a smooth projective curve defined over $\CC$.
Then $V(C)$ carries the structure of a  $1$-dimensional complex manifold, that is, a compact Riemann surface. The underlying real $2$-dimensional differential manifolds of Riemann surfaces are oriented and classified by their genus $g$, that is, by their number of handles, see Figure~\ref{fig:genus}.

\begin{figure}[htbp]
    \centering
    \includegraphics[height=2cm]{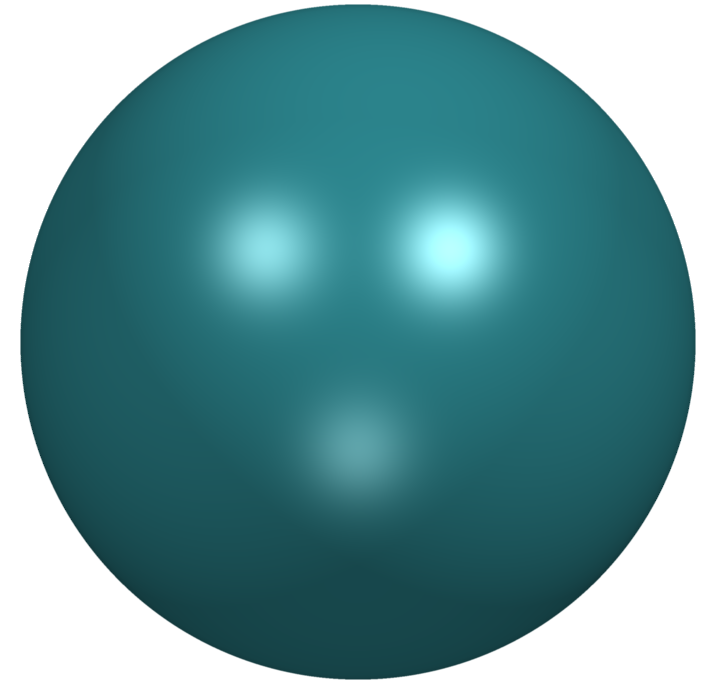}\hfill
    \includegraphics[height=2cm]{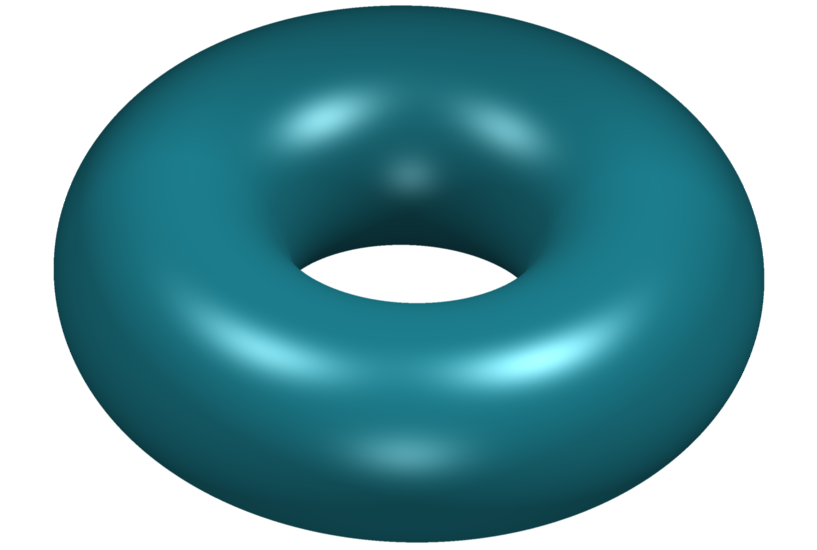}\hfill
    \includegraphics[height=2cm]{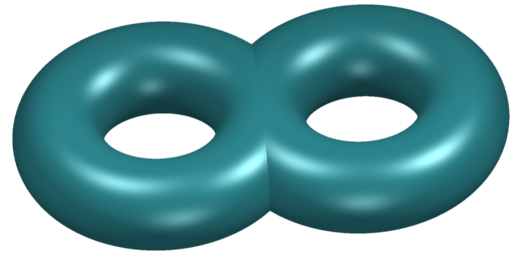}
    \caption{Illustration of Riemann surfaces of genus zero, one, and two.}
    \label{fig:genus}
\end{figure}

A rational parametrization gives rise to a differentiable surjective finite branched covering $\PP^{1}(\CC) \to C(\CC)$. Hence such a map cannot exist 
if $g=g(C)>0$ since the preimage of a non-contractible loop in $C$ would be contractible in $\PP^{1}(\CC)$.

\begin{theorem} An irreducible projective plane curve $C \subset \PP^{2}$ has  a rational parametrization iff its geometric genus (see below) is zero.
\end{theorem}
\noindent
In the smooth case over $\CC$, the geometric genus coincides with the topological genus considered above.
We do not recall the definition of the geometric genus in full generality. Instead, we recall a formula 
which relates the geometric genus of a projective plane curve to the \emph{total delta invariant} of the curve
(compare with the previous section):

\begin{proposition}\label{prop:geom-genus} The geometric genus $p_g(C)$ of a curve $C$ of degree $d$ with isolated singularities is
\[
p_g(C) = \frac{(d-1)(d-2)}{2} - \sum_{p \in C} \delta_p(C).
\]
\end{proposition}
Note that $\delta_p(C) >0$ only for the finitely many singularities $p$ of $C$.
The first term on the right hand side is the \emph{the arithmetic genus} $p_a(C)$ of $C$ (see Section \ref{syzygies} for the general definition of the arithmetic genus).

\begin{example}
Consider the plane quintic curve with defining polynomial
\begin{align*}
f = & -3x^{5} - 2x^{4}y - 3x^{3}y^{2} + xy^{4} + 3y^{5} + 6x^{4} + 7x^{3}y\\
    &  + 3x^{2}y^{2} - 2xy^{3} - 6y^{4} - 3x^{3} - 5x^{2}y + xy^{2} + 3y^{3}
\end{align*}
as in Example \ref{ex:quintic-curve}. See Figure \ref{fig deg5} for a real picture.
\begin{figure}
[h]
\begin{center}
\includegraphics[
height=2in,
]%
{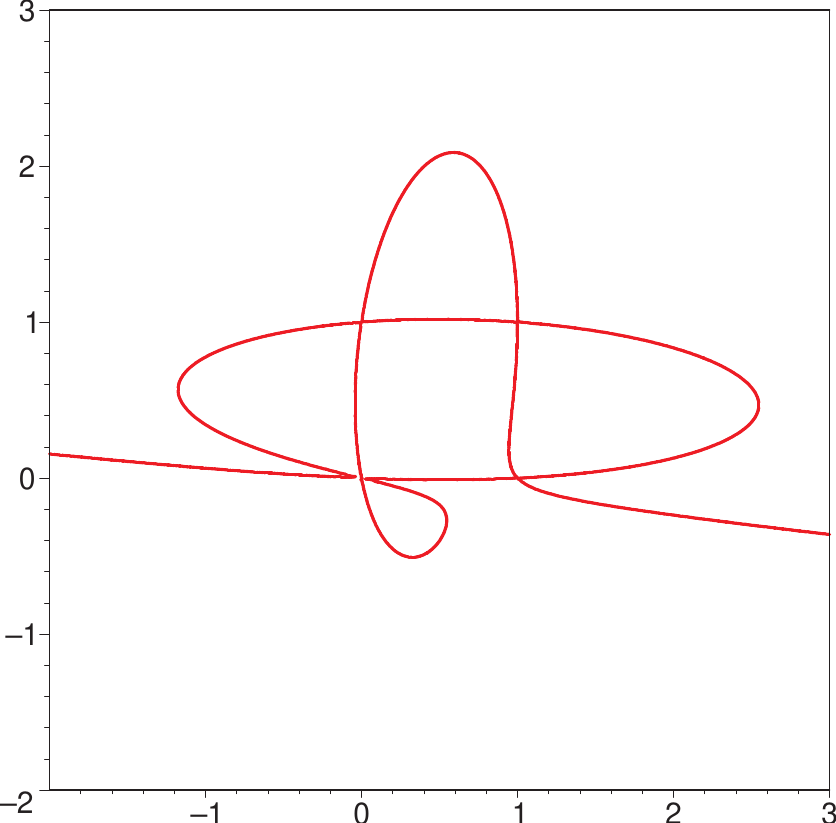}%
\caption{Degree $5$ plane curve with three double points and one triple
point.}%
\label{fig deg5}%
\end{center}
\end{figure}
The projective closure $C\subset \PP^{2}$ has a triple point at the origin, and double points at the points with coordinates $(0,1), (1,0),(1,1)$ as its only singularities. So its  geometric genus is
(see Proposition \ref{prop:geom-genus} and Example \ref{ex:quintic-curve}):

$$p_g(C)={4 \choose 2}-{3 \choose 2}-3{2 \choose 2}={4 \choose 2}-6=0.$$
 It follows that $C$ has a rational parametrization. Using B\'ezout's theorem, we can find a parametrization as follows:
Consider the pencil (one parameter family) of conics through the four singular points,
$$D_{t}=V( t(x^2-x)+y^{2}-y).$$ 
\begin{figure}
[h]
\begin{center}
\includegraphics[
height=2in,
]%
{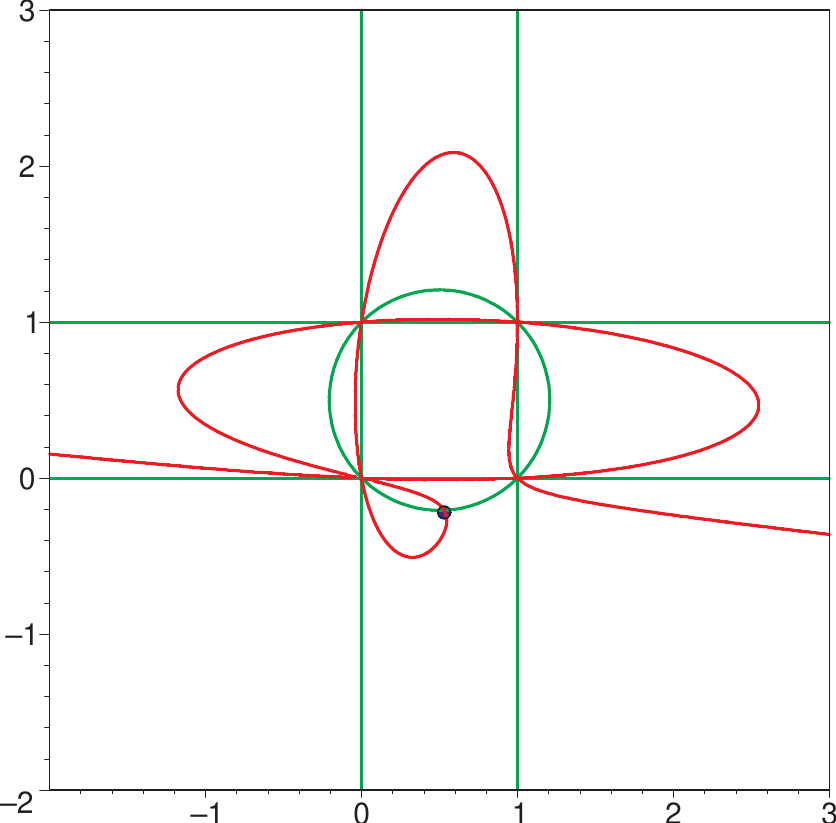}%
\caption{Three curves in the pencil of quadrics through the singular
points of $C$.}%
\label{fig deg5 linsys}%
\end{center}
\end{figure}
\noindent
Each curve $D_{t}$ %
intersects $C$ with intersection multiplicity $3$ at the origin and intersection multiplicity $2$ at the double points. Three curves in the pencil (corresponding to $t=0, \infty, 1$) are depicted in Figure~\ref{fig deg5 linsys}. 
 Since
$$2\cdot 5-3-2-2-2=1,$$
B\'ezout's theorem implies that there is exactly one moving intersection point $p(t)$. In the figure, for $t=1$, this intersection point is shown. Computing the coordinates of this point %
gives a rational parametrization. We observe that for both $t$ and $C$ we have to consider points at infinity. This motivates us to consider parametrizations as maps in the projective setting.

\begin{minted}{jlcon}
julia> P, (x, y, z) = graded_polynomial_ring(QQ, ["x", "y", "z"]);

julia> f = x^5 + 10*x^4*y + 20*x^3*y^2 + 130*x^2*y^3 - 20*x*y^4 + 20*y^5 - 2*x^4*z - 40*x^3*y*z - 150*x^2*y^2*z - 90*x*y^3*z - 40*y^4*z + x^3*z^2 + 30*x^2*y*z^2 + 110*x*y^2*z^2 + 20*y^3*z^2;

julia> C = plane_curve(f)
Projective plane curve
  defined by 0 = x^5 + 10*x^4*y - 2*x^4*z + 20*x^3*y^2 - 40*x^3*y*z + x^3*z^2 + 130*x^2*y^3 - 150*x^2*y^2*z + 30*x^2*y*z^2 - 20*x*y^4 - 90*x*y^3*z + 110*x*y^2*z^2 + 20*y^5 - 40*y^4*z + 20*y^3*z^2

julia> conics = [x^2-x*z, y^2-y*z];

julia> BM = invert_birational_map(conics, C);

julia> phi = BM["inverse"]
3-element Vector{QQMPolyRingElem}:
 -10*y(1)^5 - 430*y(1)^4*y(2) - 5020*y(1)^3*y(2)^2 - 15100*y(1)^2*y(2)^3 - 4800*y(1)*y(2)^4 - 400*y(2)^5
 y(1)^5 + 50*y(1)^4*y(2) + 690*y(1)^3*y(2)^2 + 1620*y(1)^2*y(2)^3 - 1800*y(1)*y(2)^4 - 400*y(2)^5
 y(1)^5 - 60*y(1)^4*y(2) - 2240*y(1)^3*y(2)^2 - 18100*y(1)^2*y(2)^3 - 4800*y(1)*y(2)^4 - 400*y(2)^5
    
julia> evaluate(defining_equation(C), phi)
0
\end{minted}

\end{example}

For applications it would be nice to have an algorithm for computing rational parametrizations which, ideally, are defined 
over the given field of definition of $C$. The proof of the theorem below yields such an algorithm: 

\begin{theorem}[Hilbert and Hurwitz]\label{thmHH} Let $C \subset \PP^{2}=\PP^{2}(K)$ be an irreducible curve of degree $d$ with defining equation over $k$. 
If $d$ is odd, then there exists a rational parametrization of $C$ defined over $k$. If $d$ is even, then there exists a rational 
parametrization defined over an algebraic extension field $k(a)$ of $k$ of degree $[k(a):k]=2$.
\end{theorem}

Note that the quadratic field extension is not needed if a particular conic computed by the algorithm has a $k$-rational point. 
A considerably enhanced variant of the algorithm is used by \OSCAR.

\begin{example}
The idea of Theorem~\ref{thmHH} is to determine the adjoint ideal of $C$, see \cite{adjoint}. The degree $d-2$ part of this ideal determines a rational map to a rational normal curve in $\PP^{d-2}$. We then iteratively use the anti-canonical map to arrive either at $\PP^{1}$ or a plane conic. These computations do not leave the field of definition. In the latter case, one can then decide, whether the conic has a rational point. Note that in the above example, this step is not required since the curve has odd degree. The corresponding algorithms are implemented in \OSCAR in the commands \mintinline{jl}{adjoint_ideal}, and \mintinline{jl}{rational_point_conic}:
\begin{minted}{jlcon}
julia> I = adjoint_ideal(C)
Ideal generated by
  y^3 - y^2*z
  x*y^2 - x*y*z
  x^2*y - x*y*z
  x^3 - x^2*z

julia> D = Oscar.map_to_rational_normal_curve(C)
Projective curve
  in projective 3-space over QQ with coordinates [y(1), y(2), y(3), y(4)]
defined by ideal with 3 generators

julia> betti(free_resolution(defining_ideal(D)))
       0  1
-----------
2    : 3  2
-----------
total: 3  2

julia> Oscar.rat_normal_curve_anticanonical_map(D)
2-element Vector{MPolyDecRingElem{QQFieldElem, QQMPolyRingElem}}:
 2*y(2) + 13*y(4)
 y(4)

julia> rational_point_conic(plane_curve(y^2 - x*z))
3-element Vector{QQMPolyRingElem}:
 0
 0
 -1
\end{minted}
\noindent
The command \mintinline{jl}{parametrization} combines all steps for finding a rational parametrization into one command.
\begin{minted}{jlcon}
julia> I = parametrization(C)
3-element Vector{QQMPolyRingElem}:
 25*s^5 - 1025*s^4*t + 14825*s^3*t^2 - 85565*s^2*t^3 + 146420*s*t^4 - 20780*t^5
 25*s^5 - 1400*s^4*t + 30145*s^3*t^2 - 305650*s^2*t^3 + 1396410*s*t^4 - 2023972*t^5
 25*s^5 - 1025*s^4*t + 15575*s^3*t^2 - 116205*s^2*t^3 + 562440*s*t^4 - 1898652*t^5
\end{minted}
\end{example}

\section{Syzygies}\label{syzygies}
Write $R=k[x_{1},\ldots,x_{n}]$, and let $F=R^{s}$ be the free $R$-module with basis $e_{1},\ldots ,e_{s}$. Elements of $F$ are simply column vectors
$$ \begin{pmatrix} p_{1} \\ \vdots \\ p_{s} \end{pmatrix}=\sum_{i=1}^{s} p_{i}e_{i},$$
with polynomial entries $p_{i} \in R.$ In the same spirit, we will usually not distinguish between an $s\times r$-matrix $A$ with polynomial entries $a_{i, j} \in R$
and the corresponding $R$-module homomorphism $\varphi_{A}\colon R^{r} \to R^{s}$. 
\begin{remark}
In \OSCAR,  we have a different convention:   Free module elements are thought of as row vectors, and matrices operate by multiplication on the right.
This explains the use of \mintinline{jl}{transpose} in Example \ref{ex:syz-int} below.
\end{remark}

A \emph{syzygy} on $f_{1},\ldots,f_{r} \in F$ is a tuple $(g_{1},\ldots,g_{r})^{t} \in R^{r}$ such that $\sum_{i=1}^{r} g_{i}f_{i}=0$.
That is, the syzygies are precisely the elements of the kernel of the $R$-module homomorphism $\varphi_{A}\colon R^{r} \to R^{s}$ defined by the $s\times r$-matrix $A$ with columns $f_{1},\ldots,f_{r}$.
This kernel is a finitely generated $R$-module which is called the \emph{syzygy module} of $f_{1},\ldots,f_{r}$. An $r\times t$ matrix $B$ whose columns generate  the kernel is called a \emph{syzygy matrix} of $A$.
 
The computation of syzygies has plenty of applications. For instance:
 
 \begin{proposition} Let $I=(f_{1},\ldots,f_{r})$ and $J=(g_{1},\ldots,g_{s}) \subset R$ be two ideals. Consider the $2\times (r+s+1)$ matrix
 $$A=\begin{pmatrix} 
 1 & f_{1} & \ldots & f_{r} & 0 &\ldots & 0 \\
 1 & 0 &\dots & 0 & g_{1} & \ldots & g_{r}
 \end{pmatrix}.$$
  Let 
  $B=(b_{ij})$
be a syzygy matrix of $A$ of size $(r+s+1)\times t$, say. Then the ideal $I\cap J \subset R$ is generated by the entries of the first row of $B$.
 \end{proposition}

 \noindent
We leave it to the reader to formulate a recipe for computing \emph{colon ideals} $$I:J = \{ f \in k[x_1, \dots, x_n] \mid fg\in I \text{ for all } g\in J\}.$$
Iteratively, this allows us to compute the \emph{saturation}
 $$
 I:J^{\infty} := \{ f \in k[x_1, \dots, x_n]\mid f J^m \subset I {\text{ for some }}
m\geq 1 \} = \bigcup_{m=1}^{\infty} (I:J^m)
 $$
 
\begin{example}\label{ex:syz-int} We have
$$(x_{0},x_{1})\cap (x_{2},x_{3})=(x_{0}x_{2},x_{1}x_{2},x_{0}x_{3},x_{1}x_{3})\subset \QQ[x_{0},\ldots, x_{3}].$$
Indeed, setting
$$A=\begin{pmatrix}
       1&x_{0}&x_{1}&0&0\\
       1&0&0&x_{2}&x_{3}
       \end{pmatrix},$$
       we get:

\begin{minted}{jlcon}
julia> R, (x0, x1, x2, x3) = polynomial_ring(QQ, ["x0", "x1", "x2", "x3"]);

julia> I = ideal([x0, x1]);

julia> J = ideal([x2, x3]);

julia> A = matrix(R, [1 x0 x1 0 0; 1 0 0 x2 x3])
[1   x0   x1    0    0]
[1    0    0   x2   x3]

julia> B = transpose(syz(transpose(A)))
[  0     0   x1*x3   x0*x3   x1*x2   x0*x2]
[  0   -x1       0     -x3       0     -x2]
[  0    x0     -x3       0     -x2       0]
[-x3     0       0       0     -x1     -x0]
[ x2     0     -x1     -x0       0       0]

julia> intersect(I,J)
Ideal generated by
  x1*x3
  x0*x3
  x1*x2
  x0*x2
\end{minted}
\end{example} 

Syzygies play a particular important role when it comes to Buchberger's algorithm. We discuss this algorithm in the general context of submodules of free modules. We extend some of our terminology from ideals to this case. 
\begin{definition}
A \emph{monomial} in $F$ is the product $x^{\alpha}e_{i}$ of a monomial $x^{\alpha}\in R$ and a basis vector $e_i$ of $F$, a \emph{term} in $F$ is a scalar times a monomial. A \emph{monomial ordering} $>$ on $F$ is a total ordering $>$ on the set of monomials satisfying
$$x^{\alpha}e_{i} > x^{\beta}e_{j} \;\Longrightarrow\;x^{\alpha+\gamma}e_{i} > x^{\beta+\gamma}e_{j}$$
for any triple consisting of two monomials $x^{\alpha}e_{i}$, $x^{\beta}e_{j}$ in $F$ and a monomial $x^{\gamma}$ in $R$.  For simplicity, we require in addition that
$$ 
x^\alpha e_i > x^\beta e_i \iff x^\alpha e_j > x^\beta e_j 
\;\hbox{ for all }\; i,j.
$$
Then $>$ induces a unique monomial ordering on $R$ in the obvious way, 
and we say that  $>$ is {\emph{global}} if the induced ordering on $R$ is global. 
\end{definition}
\textbf{For the rest of this section}, $>$ will be a global monomial ordering on $F$.
Unless otherwise stated, all our considerations will be relative to this ordering.

If $0\neq f= \sum f_{\alpha,i} x^{\alpha}e_{i} \in F$ is written as the finite sum of its nonzero terms, 
then the  \emph{leading term} of $f$ with respect to $>$ is
$$\Lt(f) \;\!:=\;\! \Lt_>(f) \;\!:=\;\! f_{\beta,j}x^{\beta}e_{j}, \hbox{  where }
x^{\beta}e_{j} = \max\{ x^{\alpha}e_{i} \mid f_{\alpha,i}\not= 0 \}.$$
The notions \emph{leading module} and \emph{Gröbner basis} of a submodule $I \subset F$ extend similarly from the ideal case.

\begin{theorem}[Division With Remainder in Free Modules]\label{thm:div-free-modules} Let $f_{1},\ldots,f_{r} \in F\setminus \{0\}$
be given. For each $f \in F$, there is a unique expression $f=g_{1}f_{1}+\ldots g_{r}f_{r} +h$, with
$g_{1},\ldots, g_{r} \in R$ and a \emph{remainder} $h\in F$, and such such that:
\begin{enumerate}
\item\label{thm:divrem-1} for $i>j$, no term of $g_{i}\Lt(f_{i})$ is a multiple of $\Lt(f_{j})$,
\item\label{thm:divrem-2}  for all $i$, no term of $h$ is a multiple of $\Lt(f_{i})$.
\end{enumerate}
\end{theorem} 

Let now $I =(f_1, \dots, f_r)\subset F$ be a submodule with given generators $f_i$. Consider the \emph{induced monomial ordering} $>_{1}$ on the free $R$-module $F_{1}:=R^r$ with basis $e^{(1)}_1, \ldots, e^{(1)}_r$:
\begin{align*}
x^{\alpha}e^{(1)}_{i} >_{1} x^{\beta}e^{(1)}_{j}\;  :\Longleftrightarrow\; & x^{\alpha}f_{i} >_{0} x^{\beta}f_{j} \hbox{ or } \\
&( x^{\alpha}f_{i} = x^{\beta}f_{j} \text{ up to a scalar\footnotemark \ and $i>j$}).
\end{align*}
\footnotetext{Note that the phrase ``up to a scalar'' would be superfluous, if we assume that the leading coefficients of the $f_{i}$ are all equal to $1$.}
Then $>_1$ is global since $>$ is global. 

Starting from some set of generators for a given submodule of a free module, the basic idea of Buchberger's algorithm for computing Gröbner bases 
is to consider at each step the remainders of differences of monomial multiples of two elements of the intermediate generating set, chosen so that 
the original leading terms cancel. If a remainder is nonzero, it is added to the generating set. 
To make this an algorithm, a criterion for termination is required.

To avoid superfluous \emph{Buchberger tests}, we consider the following monomial ideals: For a given set $f_{1},\ldots, f_{r}$ of generators of $I$, set
$$M_{i}=(\Lt(f_{1}),\ldots,\Lt(f_{i-1})):\Lt(f_{i}), \hbox{ for } i=2,\ldots, r.$$
Here, $I:J=\{g \in R \mid gJ \subset I \}$ denotes the colon ideal of two submodules $I,J$ of $F$.
Notice that if $x^{\alpha}\in M_{i}$, then the element $x^{\alpha}f_{i}$ is not allowed in a division expression by condition \ref{thm:divrem-1}  
of Theorem \ref{thm:div-free-modules} above.

\begin{theorem}[Buchberger's criterion] Let $I=(f_{1},\ldots,f_{r})\subset F$ be a submodule. Then $f_{1},\ldots, f_{r}$ is a Gröbner basis of $I$ iff for each $i$ and each minimal generator $x^{\alpha}$ of the monomial ideal 
$M_{i}$ above, the remainder upon division of $x^{\alpha}f_{i}$ by $f_{1},\ldots, f_{r}$ is zero.
\end{theorem}

\begin{proof} The condition is clearly necessary. To prove that it is sufficient, we consider a division expression 
$x^{\alpha}f_{i}= \sum_{j=1}^{r} g^{(i,\alpha)}_{j} f_{j}$ with remainder zero  which satisfies condition \ref{thm:divrem-1}   of Theorem \ref{thm:div-free-modules}. Then
$$G^{(i,\alpha)}= (-g^{(i,\alpha)}_{1}, \ldots, x^{\alpha}-g^{(i,\alpha)}_{i}, \ldots , -g^{(i,\alpha)}_{r})^{t} \in F_{1}=R^{r}$$
is a syzygy on $f_{1}, \ldots, f_{r}$.
With respect to the induced monomial ordering $>_{1}$ on $F_{1}$, we have $\Lt_{>_1}(G^{(i,\alpha)})= x^{\alpha}e^{(1)}_{i}$. Indeed, precisely one leading term 
$\Lt_>(g_{j}^{(i,\alpha)})f_{j}$ coincides with $x^{\alpha}\Lt_>(f_{i})$ and this $j$ satisfies $i>j$ since condition \ref{thm:divrem-1}  of Theorem~\ref{thm:div-free-modules} 
is satisfied for the $g^{(i,\alpha)}_{j}$. All other terms of a product $g^{(i,\alpha)}_{\ell}\Lt_>(f_{\ell})$ are smaller than $x^{\alpha}\Lt_>(f_{i})$.
Now, consider an arbitrary element $f=a_{1}f_{1}+\ldots,a_{r}f_{r} \in I$, 
and let $\sum_{j=1}^{r} g_{j}e^{(1)}_{j}\in F_{1}$ be the remainder of $\sum_{j=1}^{r }a_{j}e_{j}$ on division by the $G^{(i,\alpha)}$. Then $f=\sum_{j=1}^{r} g_{j}f_{j}$ since the $G^{(i,\alpha)}$ are syzygies on $f_{1},\ldots,f_{r}$. Hence,
$$ 
\Lt_>(f) = \max\{ \Lt_>(g_{j})\Lt(f_{j}) \mid j=1,\ldots, r \} \in (\Lt_>(f_{1}),\ldots, \Lt_>(f_{r}))
$$
since these leading terms have pairwise different monomial parts.
\end{proof}

\noindent
The $G^{(i,\alpha)}$ in the proof are called \emph{Buchberger's test syzygies}.

\begin{corollary}[Schreyer]\label{cor:schreyer-syz} Consider the map $\varphi\colon F_{1}\to F$ defined by $e^{(1)}_{i} \mapsto f_{i}$. If $f_{1}, \ldots, f_{r}$ is a Gröbner basis of $I$ with respect to $>$, then Buchberger's test syzygies form a Gröbner basis of
$\ker( \varphi)$ with respect to $>_1$.
\end{corollary}

\noindent
The corollary is the starting point for Schreyer's syzygy algorithm.

\begin{theorem}[Hilbert's syzygy theorem] Every finitely generated module $M$ over $R=k[x_{1},\ldots,x_{n}]$ has a finite free resolution
$$0 \leftarrow M \leftarrow F_{0} \leftarrow F_{1} \leftarrow \ldots \leftarrow F_{c} \leftarrow 0 $$
of length $c \le n$.
\end{theorem}

\begin{proof}[Schreyer] Since $M$ is finitely generated, there exists a surjection $\psi\colon F_{0} \to M$ of a free module $F_{0}=R^{r_{0}}$ onto $M$. Choose a global monomial ordering $>_{0}$ on $F_{0}$ and compute a Gröbner basis $f_{1},\ldots, f_{r_{1}}$ of $\ker(\psi)$. 
By the corollary, Buchberger's test syzygies form a Gröbner basis of $\ker(\varphi_{1})$, where 
$\varphi_{1}\colon F_{1}=R^{r_{1}} \to F_{0}$ is defined by $e^{(1)}_{i} \mapsto f_{i}$. Next, we can consider the induced monomial ordering $>_{1}$ on $F_{1}$ and the Buchberger test syzygies among the $G^{(i,\alpha)}$ to obtain a surjection $\varphi_{2}:F_{2} \to \ker(\varphi_{1})$. Repeating this process yields a free resolution of $M$.
Note that the result of the process depends, in particular, on how the Gröbner basis elements at each step are sorted. Sort the $f_{j}$ such that $j<i$ holds if for $\Lt(f_{j})=ax^{\alpha(j)}e_{\ell(j)}$ the following condition is satisfied: 
\begin{align*}
&\ell(j) < \ell(i) \hbox{ or }&& \\ &(\ell(j) =\ell(i) \hbox{ and } & \deg x^{\alpha(j)} < \deg x^{\alpha(i)}& \hbox{ or } &\\
&&( \deg x^{\alpha(j)} = \deg x^{\alpha(i)}& \hbox{ and } &x^{\alpha(j)} >_{\drlex} x^{\alpha(i)} ))
\end{align*}
Sort the Buchberger test syzygies at each step similarly. With this specification, the process stops after at most $n$ steps. 
\end{proof}

\begin{example} \label{syzygiesOfFivePoints}
Consider the ideal $J \subset S=k[w,x,y,z]$ generated by the entries of the first column in the following table.

\vspace{0.3cm}
\begin{center}
\begin{tabular}{|c|| c | c | c | c| c| c |}\hline
${\color{red}w^{2}}-xz$ & $-x$ & $ y$ & $0$ & $-z$ & $0$ & $-y^{2}+wz$ \cr
${\color{red}wx}-yz$      & ${\color{red}w}$ & $-x$&$-y$&$0$&$z$&$z^{2}$ \cr
${\color{red}x^{2}}-wy$ & $-z$ &${\color{red}w}$ & $0$ & $-y$ & $0$& $0$\cr
${\color{red}xy}-z^{2}$ & $0$& $0$ & ${\color{red}w}$ & ${\color{red}x}$& $-y$ &$-yz$ \cr
${\color{red}y^{2}}-wz$ & $0$& $0$& $-z$& $-w$ & ${\color{red}x}$ & ${\color{red}w^{2}}$ \cr\hline\hline
& $0$ & $y$ & $-x$& ${\color{red}w}$& $-z$& $1$ \cr
& $-y^{2}+wz$ &$ z^{2}$& $-wy$& $yz$ &$-w^{2}$& ${\color{red}x}$ \cr\hline
\end{tabular}  
\end{center}

\vspace{0.3cm}\noindent
The original generators turn out to be a degree reverse lexicographic Gr\"obner basis of $J$, and the algorithm produces a free resolution of shape
$$ 
\xymatrix{0 &\ar[l] S/J&\ar[l] S& \ar[l]_{\varphi_{1}} S^{5} &\ar[l]_{\varphi_{2}} S^{6} & \ar[l]_{\varphi_{3}} S^{2} &\ar[l] 0},
$$
with matrices
\begin{tabular}{|c| c|}\hline
$\varphi_{1}^{t}$ & $\varphi_{2}$ \cr\hline
&$ \varphi_{3}^{t}$ \cr \hline
\end{tabular} 
as in the table. \end{example}

In the graded case, one keeps track of the grading. Let $S=k[x_{0},\ldots,x_{n}]$ be the standard graded polynomial ring. The twisted module
$S(e)=\oplus_{d} S_{d+e}$ is the free $S$-module with the generator $1 \in S_{0}=S(e)_{-e}$ in degree $-e$.
Requiring that the matrices in the resolution respect the grading, we obtain in the example above a graded complex
$$ 
0 \leftarrow S/J\leftarrow S\leftarrow S(-2)^{5} \leftarrow S(-3)^{5}\oplus S(-4) \leftarrow S(-4)\oplus S(-5) \leftarrow 0.
$$

We can compute this resolution and access the graded free modules and differentials of it as follows\footnote{There are different strategies for
computing free resolutions. In \OSCAR, the default way of doing this is Schreyer's algorithm which is based on Corollary \ref{cor:schreyer-syz}.}:

\begin{minted}{jlcon}
julia> R, (w, x, y, z) = graded_polynomial_ring(QQ, ["w", "x", "y", "z"]);

julia> I = ideal([w^2-x*z,w*x-y*z,x^2-w*y,x*y-z^2,y^2-w*z]);

julia> A, _ = quo(R, I);

julia> FA = free_resolution(A)
Free resolution of A
R^1 <---- R^5 <---- R^6 <---- R^2 <---- 0
0         1         2         3         4

julia> FA[1]
Graded free module R^5([-2]) of rank 5 over R

julia> FA[2]
Graded free module R^5([-3]) + R^1([-4]) of rank 6 over R

julia> FA[3]
Graded free module R^1([-4]) + R^1([-5]) of rank 2 over R

julia> map(FA,1)
R^5 -> R^1
e[1] -> (-w*z + y^2)*e[1]
e[2] -> (x*y - z^2)*e[1]
e[3] -> (-w*y + x^2)*e[1]
e[4] -> (w*x - y*z)*e[1]
e[5] -> (w^2 - x*z)*e[1]
Homogeneous module homomorphism

julia> map(FA,2)
R^6 -> R^5
e[1] -> -x*e[1] + y*e[2] - z*e[4]
e[2] -> w*e[1] - x*e[2] + y*e[3] + z*e[5]
e[3] -> -w*e[3] + x*e[4] - y*e[5]
e[4] -> z*e[1] - w*e[2] + y*e[4]
e[5] -> z*e[3] - w*e[4] + x*e[5]
e[6] -> (-w^2 + x*z)*e[1] + (-w*z + y^2)*e[5]
Homogeneous module homomorphism

julia> map(FA,3)
R^2 -> R^6
e[1] -> -w*e[2] - y*e[3] + x*e[4] - e[6]
e[2] -> (-w^2 + x*z)*e[1] + y*z*e[2] + z^2*e[3] - w*y*e[4] + (w*z - y^2)*e[5] + x*e[6]
Homogeneous module homomorphism
\end{minted}

Hilbert's original motivation for the syzygy theorem was to prove  the polynomial nature of the Hilbert function.
Let $M=\sum_{d\in \ZZ} M_{d}$ be a finitely generated graded $S$-module. The function 
$$H_{M}\colon \ZZ \to \ZZ, d \mapsto \dim_{k} M_{d},$$
is called the \emph{Hilbert function} of $M$.
If 
$$
0 \leftarrow M \leftarrow F_{0} \leftarrow F_{1} \leftarrow \ldots  \leftarrow F_{c} \leftarrow 0
$$
is a finite graded free resolution with, say, $F_{i}= \bigoplus_{j} S(-j)^{\beta_{ij}}$, then
\begin{equation}\label{hilbFormel}
H_{M}(d)= \sum_{i=0}^{c} (-1)^{i}\sum_{j}\beta_{ij}{d+n-j \choose n}
\end{equation}
since $\dim_{k} S_{d}={n+d \choose n}$.
Interpreting the binomial coefficients as a polynomial ${t+n-j \choose n}=\frac{1}{n!}\prod_{\ell=1}^{n} (t+\ell-j)$, the same formula
defines a polynomial $P_{M}(t) \in \QQ[t]$ such that $H_{M}(d)=P_{M}(d)$ for all $d\gg 0$. We call
$P_{M}(t)$  the \emph{Hilbert polynomial} of $M$.

\begin{theorem} Let $I \subset S$ be a homogeneous ideal. Then $V(I) \subset \PP^{n}$ has dimension $r$ iff the Hilbert polynomial 
$P_{S/I}(t) \in \QQ[t]$ has degree $r$. In that case, $P_{S/I}(t)$ has the form
$$P_{S/I}(t)= d \frac{t^{r}}{r!} + \hbox{ lower degree terms },$$
for some integer $d >0$.
\end{theorem}

If $A \subset \PP^{n}$ is a projective algebraic set, we apply this to the
homogeneous coordinate ring $K[A]=S/\II(A)$, calling $P_{A}(t):=P_{K[A]}(t)$ the \emph{Hilbert polynomial} of $A$. The integer $d$ is then called the \emph{degree} of $A$. This degree has a geometric interpretation: 
Bertini's theorem (see \cite{Hartshorne1977}) implies that the degree of an $r$-dimensional algebraic set coincides with the number of intersection points of $A$ 
with a general linear subspace $\PP^{n-r}$ of complementary dimension. 
The \emph{arithmetic genus} of $A$ is defined to be 
$P_a(A) = (-1)^r(P_{K[A]}(0)-1)$. 

The geometric genus $p_g$ of a smooth irreducible projective curve $C \subset \PP^{n}$ can also be read off its Hilbert polynomial:
$$P_{C}(t) = d t +1 - p_g(C),$$ 
where $d=\deg C$. This is an easy consequence of the Riemann--Roch theorem.

\begin{example} A smooth complete intersection $C$ of two quadrics in $\PP^{3}$ is a curve of degree $d=4$ and geometric genus $p_g(C)=1$. Indeed, the homogeneous coordinate ring $K[C] = S/\II(C)$
has the free resolution
$$
0 \leftarrow S/\II(C) \leftarrow S \leftarrow S(-2)^{2}\leftarrow S(-4)\leftarrow 0,
$$
so that 
$$p_{C}(t)={t+3 \choose 3}-2{t+1 \choose 3}+{t-1 \choose 3} =4t.$$
\vspace*{1mm}

In \OSCAR, we get:
\begin{minted}{jlcon}
julia> R, (w,x,y,z) = graded_polynomial_ring(QQ, ["w", "x", "y", "z"]);

julia> I = ideal(R, [x*w-y*z, y*w-(x-z)*(x-2*z)]);

julia> Q = projective_scheme(I);

julia> is_smooth(Q)
true

julia> hilbert_polynomial(Q)
4*t

julia> degree(Q)
4

julia> arithmetic_genus(Q)
1

\end{minted}
\end{example}

A finitely generated graded module $M$ has a minimal graded free resolution which is obtained by choosing a minimal set of homogeneous generators of $M$ and all its syzygy modules $\ker(\varphi_{i})$. 
 A resolution with differentials $\varphi_{i}\colon F_{i}\to F_{i-1}$ for $i\ge 1$ is minimal iff the entries of the matrices $\varphi_{i}$ are contained in the homogeneous maximal ideal $(x_{0},\ldots,x_{n})$. Such a minimal resolution for $M$ is unique up to isomorphism and the numbers $\beta_{ij}$ appearing in the resolution are called the \emph{graded Betti numbers of $M$}.
 These numbers are numerical invariants of $M$ which refine the Hilbert function and the Hilbert polynomial as described in Equation \ref{hilbFormel}.
 
 More generally, we may speak of the graded Betti numbers of an arbitrary graded free resolution. It is convenient to display these numbers in a 
\emph{Betti table} $(b_{ij})$,  with $b_{ij}=\beta_{i,i+j}$.
 
 \begin{example} The resolution of the ideal $J$ found in Example \ref{syzygiesOfFivePoints} is not minimal\footnote{Computing graded free resolutions with Schreyer's algorithm
is typically quite efficient, but the result may be far from being minimal.}: Note
 that $1$ is an entry of the matrix giving $\varphi_{3}$. The \OSCAR computation below shows the corresponding Betti table as well as the table of minimal 
Betti numbers of $J$ (the symbol \mintinline{jl}{-} refers to a zero entry):

\begin{minted}{jlcon}
julia> betti_table(FA)
       0  1  2  3
-----------------
0    : 1  -  -  -
1    : -  5  5  1
2    : -  -  1  1
-----------------
total: 1  5  6  2

julia> minimal_betti_table(FA)
       0  1  2  3
-----------------
0    : 1  -  -  -
1    : -  5  5  -
2    : -  -  -  1
-----------------
total: 1  5  5  1
\end{minted}

Note that  $\beta_{2,4}=b_{2,2}$ cancels against $\beta_{3,4}=b_{3,1}$.
 
 Currently, the fastest method to compute the minimal Betti numbers of a graded module $M$ is to compute a not necessarily
minimal resolution as in Example \ref{syzygiesOfFivePoints} above, and then use this resolution to compute the  dimensions of 
the Tor-groups 
 $$
 \beta_{ij}= \dim_{k} \Tor_{i}^{S}(M,k)_{j} = \dim_{k} (H_{i}(F_{*}\otimes_{S}k))_{j}.
$$ 
\end{example}

To find free resolutions, Schreyer's algorithm first computes a Gröbner basis of the given submodule (ideal),
starting from the given generators. In particular, the algorithm computes the syzygies on the Gröbner basis elements. For 
many  applications, however, it is necessary to compute the syzygies on the given set of generators. This requires some 
bookkeeping when performing Buchberger's algorithm:\\

\noindent
\textbf{Algorithm} (Computation of Syzygies)$\quad$ \\
\Input Vectors $f_{1},\ldots, f_{r} \in F$.\\
\Output A matrix $\psi \in R^{r\times t}$ whose columns generate the kernel of the $R$-module
homomorphism
$$\varphi: R^{r} \to F, e_{i} \mapsto f_{i}.$$

\begin{enumerate}
\item Choose a global monomial ordering on $F$ and compute a Gr\"obner basis $$f_{1},\ldots,f_{r},f_{r+1}, \ldots, f_{r'}$$ of 
$(f_{1},\ldots,f_{r})$, keeping track of the Buchberger test syzygies $G^{(i,\alpha)}.$
 \item Sort the $G^{(i,\alpha)}$ such that the test syzygies which  produced new Gröbner basis elements come first. 
\item[3.] 
The matrix with columns $G^{(i,\alpha)}$ has now shape
 $$\psi'  = \begin{pmatrix} A & B \cr
 C & D \cr
\end{pmatrix},
\hbox{ where }C=\begin{pmatrix}1 & & * \cr
 & \ddots & \cr
 0& & 1\cr
\end{pmatrix}$$ 
is an $(r'-r)\times (r'-r)$ upper triangular square matrix with all diagonal entries being $1$.
Return 
$$ \psi=B - AC^{-1}D.$$
\end{enumerate}
Note that one can compute $C^{-1}$ by applying row operations to the matrix
$(E|C)$ to obtain $(C'| E)$. The inverse matrix $C'=C^{-1}$ has entries in $R$.\\

The compution of kernels of morphisms between modules is central to many homological constructions in \OSCAR.\\

\noindent
\textbf{Algorithm} (Computation of Kernels)$\quad$\\ %
\Input An $R$-module homomorphism $\varphi\colon M \to N$ between finitely presented $R$-modules given by a commutative diagram 
of type 
\hspace{0.55cm} 
$$
\xymatrix{ R^{r_{1}}\ar[d]^{\varphi_{1}} \ar[r]^{\phi} & R^{r_{0}}\ar[d]^{\varphi_{0}} \ar[r] & M\ar[r]\ar[d]^{\varphi} & 0 \cr
R^{s_{1}} \ar[r]^{\psi} & R^{s_{0}} \ar[r] & N\ar[r] & 0\;. \cr
}
$$
\Output A presentation matrix $C$ of $\ker(\varphi)$.\\
\begin{enumerate}
\item  Compute the syzygy matrix $\begin{pmatrix} A\cr B  \end{pmatrix}$ of $(\varphi_{0}|\psi)$.
\item  Compute the syzygy matrix $\begin{pmatrix} C\cr D  \end{pmatrix}$ of $(A|\phi)$.
\item  Then $C$ is a presentation matrix of $\ker(\varphi)$:
$$
\xymatrix{ R^{t_{1}} \ar[r]^{C} &R^{t_{0}} \ar[r] & \ker(\varphi)\ar[r] & 0.  \cr
}
$$
\end{enumerate}

It is an important design feature of \OSCAR that modules are represented as subquotients rather than being given by free representations.

\begin{definition} Let $A\colon R^a \to R^{s}$ and $B \colon R^{b} \to R^{s}$ be two morphisms between free $R$-modules with the same codomain.
The \emph{subquotient} defined by $A$ and $B$ is
$$ \subquo(A,B)= \frac{ \imF(A)+\imF(B)}{\imF(B)}\cong \frac{\imF(A)}{\imF(A)\cap\imF(B)}.$$
In the \OSCAR manual, we refer to
\begin{itemize}
\item the common codomain $R^s$ as the \emph{ambient free module} of $M$,
\item the images of the basis vectors of $R^a$ in $R^s$ as the \emph{ambient representatives of the generators} of $M$, and
\item the images of the basis vectors of $R^b$ in $R^s$ as the \emph{relations} of $M$.
\end{itemize}
\end{definition}

\begin{example}(Touching Surfaces)\label{touchingSurfaces} Consider the polynomials
$$ f=\det \begin{pmatrix} x_{3} & x_{2} \cr x_{2} & x_{1}\end{pmatrix} \hbox{ and }
g=\det \begin{pmatrix} x_{3} & x_{2} & x_{1} \cr 
			       x_{2} & x_{1} & x_{0} \cr
			       x_{1} & x_{0} & 0 \end{pmatrix} \in S=K[x_{0},\ldots,x_{3}].$$
See Figure~\ref{fig:touching} for an illustration of the corresponding surfaces.
 
\begin{figure}[ht]
\centering
  \includegraphics[scale=0.1]{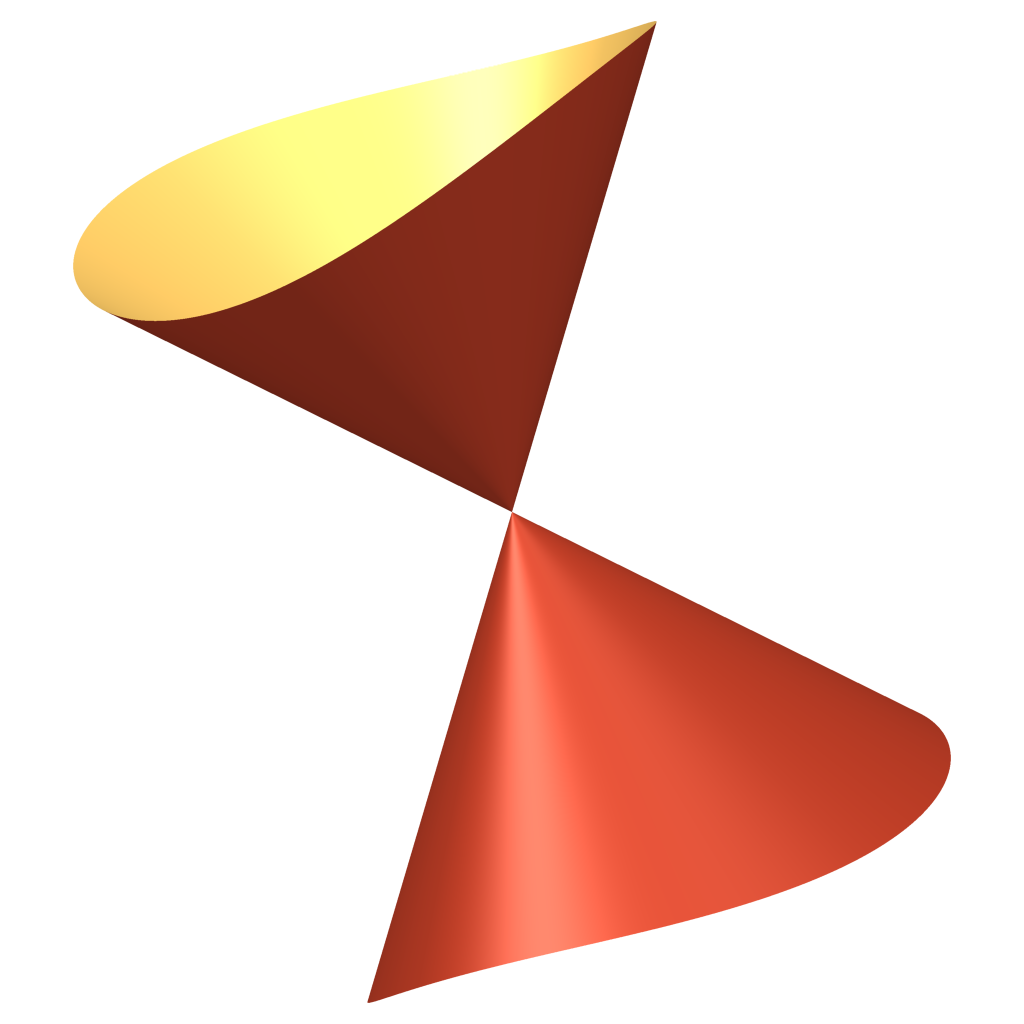}
  \includegraphics[scale=0.1]{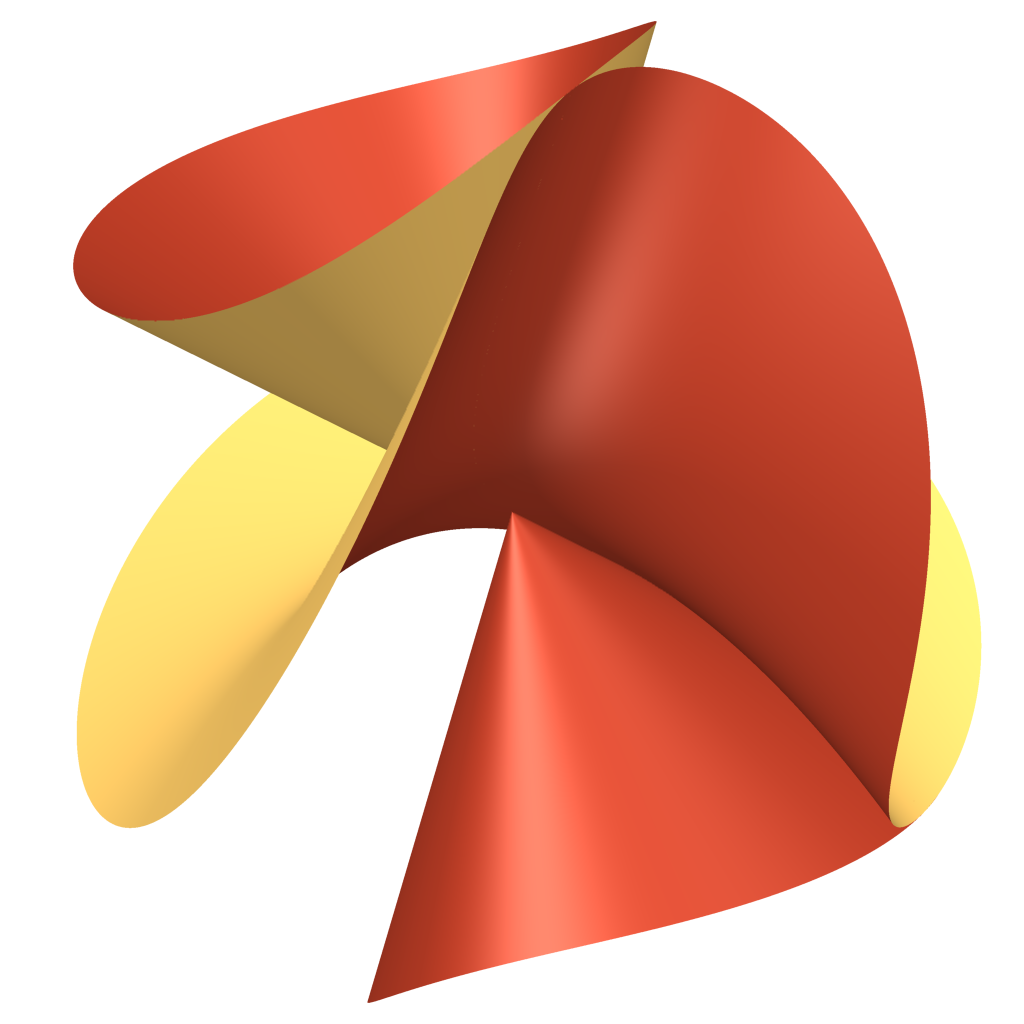}
  \includegraphics[scale=0.1]{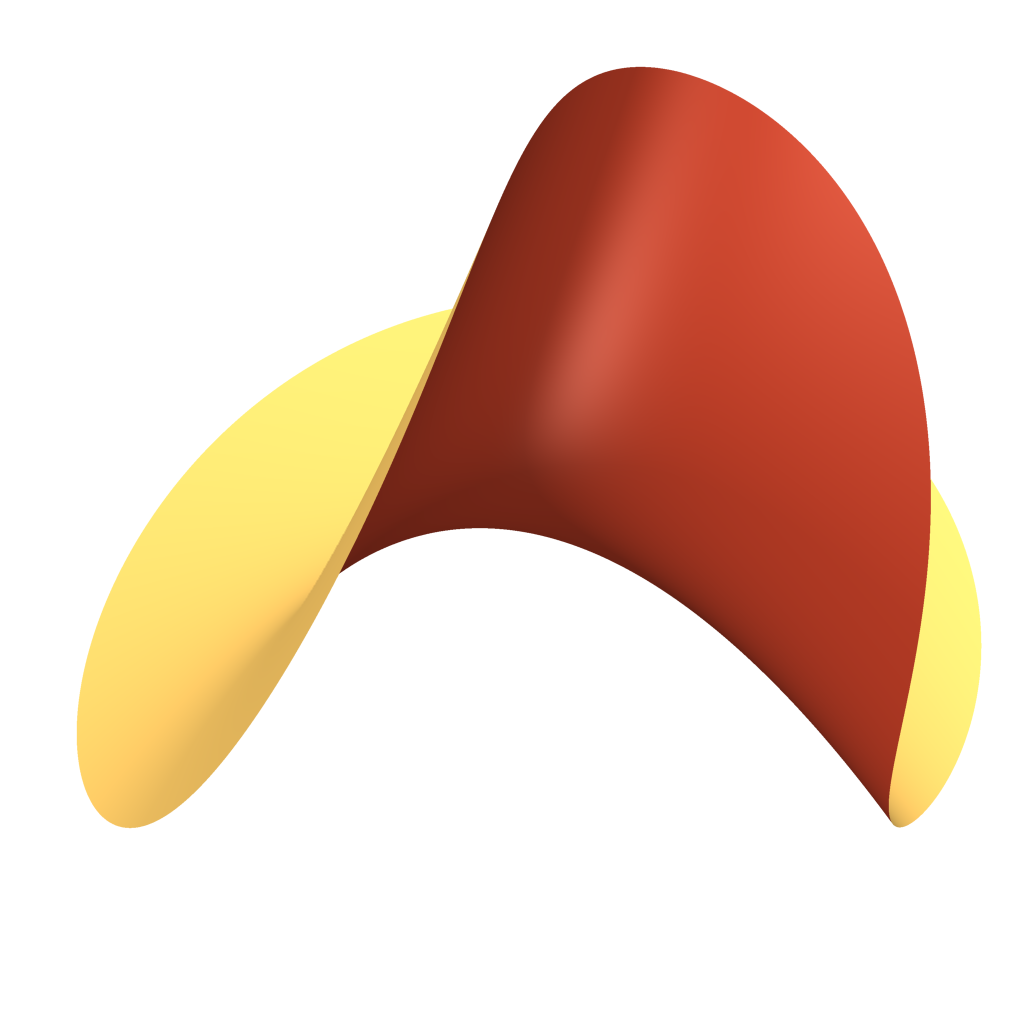}
  \caption{Touching surfaces.}
  \label{fig:touching}
\end{figure}

\noindent
With respect to $>_\degrevlex$, we have $\Lt(f) = -x_{2}^{2}$ and $\Lt(g) = -x_{1}^{3}$: 
\begin{minted}{jlcon}
julia> S, x = graded_polynomial_ring(QQ, ["x_0", "x_1", "x_2", "x_3"]);

julia> m3x3 = matrix(S, 3, 3, [(i + j == 4) ? 0 : x[4 - (i + j)] for i in 0:2, j in 0:2])
[x_3   x_2   x_1]
[x_2   x_1   x_0]
[x_1   x_0     0]

julia> f = det(m3x3[1:2, 1:2])
x_1*x_3 - x_2^2

julia> g = det(m3x3)
-x_0^2*x_3 + 2*x_0*x_1*x_2 - x_1^3

julia> leading_term(f)
-x_2^2

julia> leading_term(g)
-x_1^3
\end{minted}
\noindent
Thus, $f,g$ form a Gr\"obner basis of $J = (f,g)$, $J$ is unmixed of dimension one by Theorem~\ref{GBDimensionCriterion}, and $A = V (J )\subset \PP^{3}$ is a curve. The twisted cubic curve $C \subset \PP^{3}$, which is defined by the $2\times 2$-minors of the matrix
$$\begin{pmatrix} x_{3} & x_{2} & x_{1} \cr 
			       x_{2} & x_{1} & x_{0} 
			        \end{pmatrix},$$
is a component of $A$. Thus, $\pp_{1}=(x_{1}^{2}-x_{0}x_{2},x_{1}x_{2}-x_{0}x_{3},x_{2}^{2}-x_{1}x_{3})$ is one of the associated primes of $J$. Actually, $\rad(J)=\pp_{1}=\II(C)$:
\begin{minted}{jlcon}
julia> p1 = radical(J);

julia> p1 == ideal(minors(m3x3[1:3, 1:2],2))
true
\end{minted}
\noindent
\noindent
The intersection multiplicity of the two surfaces along $C$ is computed by counting the number of times $\pp_1$ occurs in a filtration of the module $M=S/J$ via primes (see \cite[Ch. I, Sec. 7]{Hartshorne1977}). To compute such a filtration, we first construct the graded $S$-module $\Hom(S/\pp_{1}, S/J)$:
\begin{minted}{jlcon}
julia> mp1 = ideal_as_module(p1)
Graded submodule of S^1
1 -> (-x_1*x_3 + x_2^2)*e[1]
2 -> (-x_0*x_3 + x_1*x_2)*e[1]
3 -> (-x_0*x_2 + x_1^2)*e[1]
represented as subquotient with no relations

julia> M1, _ = quo(ambient_free_module(mp1), mp1);

julia> M1
Graded subquotient of submodule of S^1 generated by
1 -> e[1]
by submodule of S^1 generated by
1 -> (-x_1*x_3 + x_2^2)*e[1]
2 -> (-x_0*x_3 + x_1*x_2)*e[1]
3 -> (-x_0*x_2 + x_1^2)*e[1]

julia> mJ = ideal_as_module(J);

julia> M, _  = quo(ambient_free_module(mJ),mJ);

julia> homM1M, psi = hom(M1, M);

julia> hom1, tohomM1M = prune_with_map(homM1M);

julia> hom1
Graded subquotient of submodule of S^2 generated by
1 -> e[1]
2 -> e[2]
by submodule of S^2 generated by
1 -> -x_2*e[1] + x_3*e[2]
2 -> x_0*e[1] - x_1*e[2]
3 -> -x_1*e[1] + x_2*e[2]
\end{minted}
\noindent
Actually, $\Hom(S/\pp_{1}, S/J)$ is generated in degree $2$ with generators corresponding to multiplication by $x_{0}x_{2}-x_{1}^{2}$ and $x_{0}x_{3}-x_{1}x_{2}$, respectively: 
\begin{minted}{jlcon}
julia> degrees_of_generators(hom1)
2-element Vector{FinGenAbGroupElem}:
 [2]
 [2]

julia> phi1 = psi(tohomM1M(hom1[1]))
M1 -> M
e[1] -> (-x_0*x_3 + x_1*x_2)*e[1]
Graded module homomorphism of degree [2]

julia> phi2 = psi(tohomM1M(hom1[2]))
M1 -> M
e[1] -> (-x_0*x_2 + x_1^2)*e[1]
Graded module homomorphism of degree [2]
\end{minted}
\noindent
As the first step towards the desired filtration, we may consider
$$M_{1} =S/\pp_{1}(-2) \hookrightarrow M=S/J,$$
where the inclusion is induced by multiplication with $x_{0}x_{2}-x_{1}^{2}$:
\begin{minted}{jlcon}
julia> iszero(kerphi2)
true
\end{minted}
\noindent
The annihilator of the image of $x_{1} x_{2}-x_{0} x_{3}$ in $M/M_{1}$ is
the prime ideal $\pp_{2} = (x_{0}, x_{1}, x_{2})$:
\begin{minted}{jlcon}
julia> MmodM1 = cokernel(phi2)
Graded subquotient of submodule of S^1 generated by
1 -> e[1]
by submodule of S^1 generated by
1 -> (x_1*x_3 - x_2^2)*e[1]
2 -> (-x_0^2*x_3 + 2*x_0*x_1*x_2 - x_1^3)*e[1]
3 -> (-x_0*x_2 + x_1^2)*e[1]

julia> p2 = ideal([x[1],x[2],x[3]])
Ideal generated by
  x_0
  x_1
  x_2

julia> f = x[2]*x[3]-x[1]*x[4]
-x_0*x_3 + x_1*x_2

julia> v = f*MmodM1[1];

julia> U, inclU = sub(MmodM1, [v]);

julia> annihilator(U) == p2
true
\end{minted} 
\noindent
 Now, we may take $M_{2}\subset M$ as the preimage under the projection map $M\rightarrow M/M_1$ of the image of the map $S/\pp_{2}(-2) \hookrightarrow M/M_{1}$ and $M_{3}=M$ because $M/M_{2} \cong S/\pp_{1}$:
\begin{minted}{jlcon}
julia> Smodp2, _ = quo(ambient_free_module(mp2), mp2);

julia> homp2M1, tau = hom(Smodp2, MmodM1);

julia> hom2, tohomp2M1 = prune_with_map(homp2M1);

julia> psi = tau(tohomp2M1(hom2[1]));

julia> iszero(kernel(psi)[1])
true

julia> annihilator(cokernel(psi))
Ideal generated by
  -x_1*x_3 + x_2^2
  -x_0*x_3 + x_1*x_2
  -x_0*x_2 + x_1^2

julia> annihilator(cokernel(psi)) == p1
true
\end{minted} 
 \noindent
 So in the filtration
 $$ 0 \subset M_{1} \subset M_{2} \subset M_{3}=M,$$
 the prime ideal $\pp_{1}$ occurs twice and $\pp_{2}$ occurs once. We conclude: The surfaces $Q = V (f )$ and $H = V (g)$ intersect with multiplicity $2$ along $C$, 
that is, they touch each other tangentially along C. 
There are no further intersections since $\deg Q \cdot \deg H = 2\cdot 3$ and $2\deg C =2\cdot 3$ coincide.
Note that $\pp_2$ is not an associated prime ideal of $M$, since $J$ is unmixed.
 In this case, it is not possible to find a filtration where
 only  associated primes of $M$ occur.
\end{example}

\section{Applications to Surfaces in $\PP^4$}

Ellingsrud and Peskine \cite{EP89} proved that there are only finitely many components of the Hilbert scheme of  $\PP^4$ 
whose general points correspond to smooth surfaces of Kodaira dimension $<2$, that is, to surfaces of non-general type.
In particular, this confirms a conjecture of Hartshorne about smooth rational surfaces in $\PP^4$. Referring to \cite{DES, DS00} 
for a unified approach to constructing non-general type surfaces in $\PP^4$, we illustrate this approach by giving two 
examples of rational surfaces in $\PP^4$. Typically, the approach requires that we find graded modules with particular syzygies.
In the first example below, the required module was found by a random computer algebra search in small characteristic. In Section 
\ref{Deformations}, we will discuss how to infer from this that smooth surfaces with the same invariants exist over $\CC$. 
For the general theory of compact complex surfaces see \cite{BPV84}.

\begin{example}[A Rational Surface in $\PP^4$ 
\cite{zbMATH00943474}]\label{ratd=11pi=10}
Consider the ideal \mintinline{jl}{m} defined below and its minimal free resolution:

\begin{minted}{jlcon}
julia> S, (x0, x1, x2, x3, x4) = graded_polynomial_ring(GF(3), ["x0", "x1", "x2", "x3", "x4"]);

julia> m = ideal(S, [x1^2+(-x1+x2+x3-x4)*x0, x1*x2+(x1-x3+x4)*x0, x1*x3+(-x1+x4+x0)*x0, x1*x4+(-x1+x3+x4-x0)*x0, x2^2+(x1-x2-x4-x0)*x0, x2*x3+(x1-x2+x3+x4-x0)*x0, x2*x4+(x1+x2-x3-x4-x0)*x0, x3^2+(x3+x4-x0)*x0,x3*x4+(-x3-x4+x0)*x0, x4^2+(x1+x3-x4-x0)*x0]);

julia> Qm, _ = quo(S, m);

julia> FQm = free_resolution(Qm, algorithm = :mres);

julia> betti_table(FQm)
       0   1   2   3   4  5
---------------------------
0    : 1   -   -   -   -  -
1    : -  10  15   2   -  -
2    : -   -   7  26  20  5
---------------------------
total: 1  10  22  28  20  5
\end{minted}
\noindent
The minimal free resolution gives rise to a subcomplex 
$$ G: 0 \leftarrow S \leftarrow S^{10}(-2) \leftarrow S^{10}(-3) \leftarrow S^{2}(-4) \leftarrow 0 $$
whose homology $H_{1}(G)\cong I_{X}(2)$  is isomorphic to the twist of the homogeneous ideal $I_{X}$ of a smooth surface $X$ of degree $d=11$ and sectional genus $\pi=11$,
with geometric genus and irregularity $p_{g}(X)=q(X)=0$. Here, the \emph{irregularity} of $X$ is the difference $q(X) = p_{g}(X) - p_{a}(X)$.

We create the surface in \OSCAR:
\begin{minted}{jlcon}
julia> phi = map(FQm, 3);

julia> M = transpose(matrix(phi)[1:2, 1:15]);

julia> DM = [map_entries(x->derivative(x,i), M) for i = 1:5];

julia> MM = transpose(hcat(DM...));

julia> size(MM)
(10, 15)

julia> NN = matrix(map(FQm, 2))[1:15, 1:10];

julia> size(NN)
(15, 10)

julia> D = graded_cokernel(transpose(MM*NN));

julia> FD = free_resolution(D, algorithm = :mres);

julia> betti_table(FD)
        0   1  2
----------------
0    : 10  10  -
1    :  -   -  1
2    :  -   -  -
3    :  -   -  1
----------------
total: 10  10  2

julia> P = cokernel(transpose(matrix(map(FD, 2))));

julia> I = annihilator(P);

julia> QI, _ = quo(S, I);

julia> FQI = free_resolution(QI, algorithm = :mres);

julia> betti_table(FQI)
       0   1   2   3  4
-----------------------
0    : 1   -   -   -  -
1    : -   -   -   -  -
2    : -   -   -   -  -
3    : -   -   -   -  -
4    : -   5   -   -  -
5    : -   7  26  20  5
-----------------------
total: 1  12  26  20  5

julia> dim(I)
3

julia> degree(I)
11
\end{minted}
\noindent
So the resulting ideal is generated by 5 quintics and 7 sextics.
It defines a surface of the desired degree. See the example of the Alexander surface below for how to check smoothness and for how to compute the sectional genus.

We would now like to determine how this surface fits into the Enriques-Kodaira classification of smooth projective surfaces. For this, we apply the adjunction process of Van de Ven and Sommese \cite{SV87}.
Recall that a smooth projective surface arises from a minimal smooth projective surface $X_{\textrm{min}}$ by repeatedly blowing-up a point. The exceptional curve $E$ of the blow-up is a so called \emph{$(-1)$-curve}. Such a curve is rational, $E \cong \PP^{1}$, with self-intersection number $E^{2}=-1$. A surface is \emph{minimal} iff it contains no $(-1)$-curves.

Let $X \subset \PP^{n}$ be a smooth projective surface over $\CC$ of codimension $c$. Let $S$ and $S_{X}$ denote the homogeneous coordinate rings of $\PP^{n}$ and $X$, respectively.
Consider $\omega_{X}=\Ext^{c}_{S}(S_{X},S(-\dim S)),$
the graded \emph{dualizing module} of $S_{X}$. A basis of $(\omega_{X})_{{1}}$
corresponds to the linear system $|K_X +H|$, where $K_X$ is a canonical divisor on $X$ and $H$ is the hyperplane class. This linear system defines a birational morphism 
$\varphi_{|K_X+H|}\colon X \to X'$ onto another smooth projective surface $X'$ which blows down precisely all $(-1)$-lines on $X$ unless
\begin{enumerate}
\item $X \subset \PP^{n}$ is a linearly or quadratically embedded $\PP^{2}$ or $X$ is ruled by lines; hence $|K_X+H| = \emptyset$ holds,
\item $X\subset \PP^{n}$ is an anti-canonically embedded del Pezzo surface, in which case $\varphi_{|K_X+H|}$ maps $X$ to a point,
\item $X\subset \PP^{n}$ is a conic bundle, in which case  $\varphi_{|K_{X}+H|}\colon X \to B$ maps $X$ to a curve $B$ and the fibers of $\varphi_{|K_{X}+H|}$ are the conics, or
\item\label{badLuck} $X\subset \PP^{n}$ is a surface in one of four families identified by Sommese and Van de Ven \cite{SV87}, and $\varphi_{|K_X+H|}\colon X \to X'$ is not birational, but finite to one.
\end{enumerate}

Unless we are in one of the exceptional cases, a $(-1)$-conic $C$ in $X$ is mapped to a $(-1)$-line in $X'$ because $(K_X+H)\;\!. \;\! C=-1+2=1$.
Thus, continuing this \emph{adjunction process}, we eventually arrive at a minimal model $X_{\textrm{min}}$ of $X$ unless $X$ has negative Kodaira dimension. 
In the rational case, we finally arrive at $\PP^{2}$,
the Veronese surface, a Hirzebruch surface, a Del Pezzo surface, a conic bundle, or, with bad luck, in one of the exceptional cases of part (\ref{badLuck}) above.

For the surface under consideration, the adjunction process produces  a sequence of birational maps
$$
\begin{matrix}
\PP^{4} &&  \PP^{9} && \PP^{10} && \PP^{8} && \PP^{5} \\
\cup     && \cup         && \cup      && \cup && \cup  \\
X^{11}_{10} &{\buildrel 3 \over \rightarrow}& X^{19}_{11} &{\buildrel 2 \over \rightarrow}& X^{18}_{9}&{\buildrel 0 \over \rightarrow}& X^{13}_{6} &{\buildrel 3 \over \rightarrow}& X^{6}_{2}
\end{matrix}
$$
Here, $X^{d}_{\pi}$ denotes a surface of degree $d$ with sectional genus $\pi$, and ${\buildrel \ell \over \rightarrow}$ denotes a birational map which blows down $\ell$ $(-1)$-curves.
The final surface $X^{6}_{2} \subset \PP^{5}$ is the complete intersection of the Segre product $\PP^{1}\times \PP^{2}$ with a quadric $Q$. Thus   $X^{6}_{2}$ is a conic bundle which has $6$ singular fibers. Blowing down one of the two $(-1)$-curves in each singular fiber, we arrive at a Hirzebruch surface. In general, the choice of curves in such a  process is not canonical, and we might need to extend our field of definition to define the maps and, consequently, a rational parametrization of $X$. Luckily, in the example under consideration, no field extension is needed. The resulting parametrization is by forms of bi-degree $(10,11)$ on $\PP^1\times \PP^1$
with $14$ base points of multiplicities $(5^6,4^4,2^2,1^3)$.
For a tutorial containing the complete \OSCAR code for this example and further explanations, we refer to the \OSCAR webpage.
\end{example}

\begin{example}[The Alexander Surface \cite{alexander1988}] Consider a module $M$ with minimal free presentation $S^{3}(-1)\oplus S^{15}(-2) \to S\oplus S^{3}(-1) \to M \to 0$ over $S=\QQ[x_{0},\ldots,x_{4}]$.
The minimal Betti table of $M$ is

\[
\begin{tightarray}{l c r r r r r r}
   & & 0 & 1 & 2 & 3 & 4 & 5 \\
\hline
0 &:   & 1 & 3 & 3 & 1 & - & - \\
1  &: & 3 & 15 & 26 & 15 & - & - \\
2  &: & - & - & - & 6 & 10 & 3 \\
\hline
\text{total}&: & 4 & 18 & 29 & 22 & 10 & 3 \\
\end{tightarray}
\]      
and the dual of the linear strand $0\to S^{6}(-5) \to S^{10}(-6) \to S^{3}(-7)\to 0$ is a complex $G$ whose homology $H_{1}(G) \cong I_{X}(9)$ 
is the twisted homogeneous ideal of a smooth projective surface $X$.  Rather than showing this construction, we demonstrate
how to load  the surface as a projective scheme from the corresponding \OSCAR database which offers preconstructions of all surfaces 
listed in \cite{DES} as well as some of the other surfaces found
later\footnote{To ease subsequent computations, the surfaces are
  constructed over finite fields.
Note, however, that in contrast to the surface in Example \ref{ratd=11pi=10}, which was found by a random search in small characteristic, the surfaces in the data 
base were constructed using recipes which also work in characteristic zero. So all computations can be confirmed in characteristic zero, although this
may be time consuming.}:

\begin{minted}{jlcon}
julia> X = rational_d9_pi6();

julia> is_smooth(X)
true

julia> degree(X)
9

julia> S = ambient_coordinate_ring(X)
Multivariate polynomial ring in 5 variables over GF(31991) graded by 
  x -> [1]
  y -> [1]
  z -> [1]
  u -> [1]
  v -> [1]

julia> B, _ = quo(S, ideal(S, [gens(S)[1]]));

julia> Y = proj(B)
Projective scheme
  over finite field of characteristic 31991
defined by ideal (x)

julia> C = intersect(X, Y);

julia> arithmetic_genus(C)
6

julia> A = homogeneous_coordinate_ring(X);

julia> FA = free_resolution(A);

julia> minimal_betti_table(FA)
       0   1   2   3  4
-----------------------
0    : 1   -   -   -  -
1    : -   -   -   -  -
2    : -   -   -   -  -
3    : -   -   -   -  -
4    : -  15  26  15  3
5    : -   1   3   3  1
-----------------------
total: 1  16  29  18  4
\end{minted}

\noindent
The sextic generator of the ideal is needed because the quintics  alone define the union of $X$ with a $6$-secant line:
\begin{minted}{jlcon}

julia> I = defining_ideal(X);

julia> IQ = ideal([x for x in gens(I) if degree(x)[1] == 5]);

julia> J = saturation(IQ, I);

julia> degree(J)
1

julia> M = I + J;

julia> degree(M)
6

\end{minted}\vspace{-2mm}
\noindent
The astute reader will notice that the linear strand 
$$0\to S^{6}(-9) \to S^{3}(-8) \to S^{3}(-7) \to S(-6)\to 0$$
arising from the resolution \mintinline{jl}{FA} is (up to twist) the Koszul complex  of a line in $\PP^4$. It should not
come as a surprise that this line is the 6-secant above.

This time, the adjunction process yields
$$
\begin{matrix}
\PP^{4} &&  \PP^{5} && \PP^{5} && \PP^{4} && \PP^{2} \\
\cup     && \cup         && \cup      && \cup && \parallel  \\
X^{9}_{6} &{\buildrel 0 \over \rightarrow}& X^{10}_{6} &{\buildrel 0 \over \rightarrow}& X^{9}_{5}&{\buildrel 0 \over \rightarrow}& X^{6}_{3} &{\buildrel 10 \over \rightarrow}& X^{1}_{0}
\end{matrix}.
$$ 
Recursively substituting the parametrizations of the intermediate surfaces in the adjunction matrices and computing the syzygies of the transposed matrices yields a rational  parametrization of $X$.
In the example, $X$ is parametrized by forms of degree $13$ which vanish to order four at ten base points.

The module $M$ has Hilbert function $(1,5,3,0,\ldots)$. A general
module with this Hilbert function has syzygies of type
   \[
\begin{tightarray}{l c r r r r r r}
    & & 0 & 1 & 2 & 3 & 4 & 5 \\
\hline
0 &:   & 1 & - & - & - & - & - \\
1  &: & - & 12 & 25 & 15 & - & - \\
2  &: & - & - & - & 6 & 10 & 3 \\
\hline
\text{total}&: & 1 & 12 & 25 & 21 & 10 & 3 \\
\end{tightarray}
\]   
\noindent
and a surface obtained as above from such a module has syzygies of type\pagebreak[4]
\[
\begin{tightarray}{l c r r r r r r}
    & & 0 & 1 & 2 & 3 & 4 \\
\hline
0 &:   & 1 & - & - & - & - \\
1  &: & - & - & - & - & - \\
2  &: & - & - & - & - & - \\
3  &: & - & - & - & - & - \\
4  &: & - & 15 & 25 & 12 & - \\
5  &: & - & - & - & - & 1 \\
\hline
\text{total}&: & 1 & 15 & 25 & 12 & 1 \\
\end{tightarray}
\]

\noindent
This time, the resulting surface has no $6$-secant line, instead it has a $(-1)$-line.
The adjunction process 
$$
\begin{matrix}
\PP^{4} &&  \PP^{5} && \PP^{5} && \PP^{5} \\
\cup     && \cup         && \cup      && \cup   \\
Y^{9}_{6} &{\buildrel 1 \over \rightarrow}& Y^{10}_{6} &{\buildrel 0 \over \rightarrow}& Y^{10}_{6}&{\buildrel 0 \over \rightarrow}& Y^{10}_{6} 
\end{matrix}
$$ 
becomes periodic, because $Y'=Y^{10}_6$ is a minimal Enriques surface and
$2K_{Y'}\sim 0$.

Notice that the twelve quadric generators of the ideal $\widetilde m$ with $M=S/\widetilde m$ are annihilated by three quadrics
in $k[\partial_0,\ldots,\partial_4]$, where
$\partial_i=\frac{\partial}{\partial x_i}$. The three quadrics  generate the homogeneous ideal to a canonical curve of genus $5$ in $\check \PP^4$. Conversely, the curve determines the ideal $\widetilde m$ and hence $M$ and the non-minimal Enriques surface $Y$.
This implies the curious fact that an open part of the universal family $\cU \to \cM_E$ of Fano polarized Enriques surfaces is isomorphic to an open part of the moduli space $\cM_5$ of genus $5$ curves. This was first observed in [DES], but awaits a geometric explanation till today.\vspace{-1mm}
\end{example}

\section{Cohomology of coherent sheaves}\label{cohomology}

The \OSCAR command \mintinline{jl}{sheaf_cohomology} offers two algorithms for computing the cohomology of coherent sheaves over projective $n$-space. 
The algorithms are based on local cohomology (see \cite{Eis98}) and on Tate resolutions via the Bernstein--Gelfand--Gelfand correspondence as introduced 
in  \cite{EFS03}, respectively. The first algorithm makes use of syzygy computations over the symmetric algebra, and the second 
algorithm is based on syzygy computations over the exterior algebra (see \cite{DE02} for a tutorial). Thus, in most examples, the 
second algorithm is much faster. The command \mintinline{jl}{sheaf_cohomology} takes as input a finitely generated graded module
whose sheafification gives the desired coherent sheaf as well as numerical information on what cohomology groups should be considered.
It returns the vector space dimensions of these groups in the form of a cohomology table.
\begin{example}\label{cohomOfIX} For the ideal sheaf  $\cI_{X}$ of the surface $X \subset \PP^{4}$ from Example \ref{ratd=11pi=10}, 
the dimensions $h^{i}(\PP^{4},\cI_{X}(j))$ in the range $j=-2,\ldots 8$ are:
\begin{minted}{jlcon}
julia> MI = ideal_as_module(I);

julia> sheaf_cohomology(MI, -2, 8, algorithm = :loccoh)
twist:   -2   -1    0    1    2    3    4    5    6    7    8
-------------------------------------------------------------
4:        -    -    -    -    -    -    -    -    -    -    -
3:       30   10    -    -    -    -    -    -    -    -    -
2:        -    -    -    2    -    -    -    -    -    -    -
1:        -    -    -    -    1    5    5    -    -    -    -
0:        -    -    -    -    -    -    -    5   32   84  170
-------------------------------------------------------------
chi:     30   10    -    2    1    5    5    5   32   84  170
\end{minted}
\noindent
Here, as for Betti tables, the symbol \mintinline{jl}{-} refers to a zero entry. For example, the $h^{1}(\PP^{4},\cI_{X}(j))$ in the range $j=1,\ldots 5$ are
$0,1,5,5,0.$
\end{example}

Now we come to the Tate resolution. Let $\cF= \widetilde M$ be the coherent sheaf on $\PP^{n}$ associated to a finitely generated graded module 
$M=\sum M_{d}$ over $S = k[x_{0},\dots,x_{n}]$. We describe how 
to compute the dimensions $h^{i}(\PP^{n},\cF(j))$ via exterior syzygies over 
the  algebra $E=k[e_{0},\ldots,e_{n}]$ dual to $S$.

Note that $M$ defines a complex of graded free $E$-modules
$$R(M): \ldots \to \Hom_{k}(E,M_{d-1}) \to \Hom_{k}(E,M_{d}) \to \Hom_{k}(E,M_{d+1}) \to \ldots. $$
The complex $R(M_{\ge r})$ is acyclic if $r$ is greater than the Castelnuovo-Mumford regularity of $M$. For such an $r$ we extend $R(M_{\ge r})$ by computing a minimal free resolution 
$$\ldots \to T^{r-2} \to T^{r-1}$$
of the $E$-module
$P=\ker(\Hom_{k}(E,M_{r}) \to \Hom_{k}(E,M_{r+1}))$. This yields a doubly infinite exact complex
$$ \TT=\TT(\cF): \ldots \to T^{r-2} \to T^{r-1} \to T^{r} \to T^{r+1} \to \ldots $$
whose isomorphism class depends only on the sheaf $\cF$. We call $\TT(\cF)$ the \emph{Tate resolution} of $\cF$.

\begin{theorem}[Eisenbud--Fl\o ystad--Schreyer] With notation as above,
$$\TT^{e}(\cF)=T^{e}= \sum_{i=0}^{n} \Hom_{k}(E,H^{i}(\PP^{n},\cF(e-i))).$$
\end{theorem}
Thus each cohomology group $H^{i}(\PP^{n},\cF(j))$ occurs precisely once in the complex. 

\begin{example} For the ideal sheaf from Example \ref{cohomOfIX}, we get:
\begin{minted}{jlcon}
julia> sheaf_cohomology(MI, -2, 8)
twist:   -2   -1    0    1    2    3    4    5    6    7    8
-------------------------------------------------------------
4:        -    -    -    -    -    -    -    *    *    *    *
3:        *   10    -    -    -    -    -    -    *    *    *
2:        *    *    -    2    -    -    -    -    -    *    *
1:        *    *    *    -    1    5    5    -    -    -    *
0:        *    *    *    *    -    -    -    5   32   84  170
-------------------------------------------------------------
chi:      *    *    *    *    1    5    5    *    *    *    *
\end{minted}
\noindent
With no algorithm specified, \OSCAR applies the default option \mintinline{jl}{algorithm = :bgg}. Due to the shape
of the Tate resolution, the function does then not compute all values in a given range of twists $l <h$. The missing values are indicated by 
a \mintinline{jl}{*}. To determine all values in the range $l <h$, enter \mintinline{jl}{ sheaf_cohomology(MI, l-ngens(base_ring(M)), h+ngens(base_ring(M)))}.
 \end{example}

\section{Plane Curves, their Duals, and Riemann--Roch Spaces}

Let $C=V(f) \subset \PP^{2}$ be an irreducible plane curve, given by a square-free homogeneous polynomial $f$. The space of tangent lines $\check C \subset \check \PP^{2}$ of $C$ is a curve in the dual projective space.%

\begin{example}\label{ex:dual-curve} We demonstrate how to write an \OSCAR function which computes the dual of a plane curve specified by a homogeneous polynomial:
\begin{minted}{jlcon}
function dual_curve(f::MPolyRingElem, P_dual::MPolyRing)
  P = parent(f)
  vars_P = gens(P)
  nvars_P = ngens(P)
  vars_P_dual = gens(P_dual)
  # Extend the original polynomial ring to include the variables of P_dual
  P_ext, vars_ext = polynomial_ring(base_ring(P), [[string(v) for v in vars_P]; [string(v) for v in vars_P_dual]])
  inc = hom(P, P_ext, vars_ext[1:nvars_P])
  f_ext = inc(f)
  # Compute the Jacobian matrix with respect to the original variables
  jf = transpose(jacobian_matrix(f_ext)[1:nvars_P, 1:1])
  # Form the matrix with the last 'ngens(P_dual)' variables of P_ext
  A = matrix([vars_ext[(end-ngens(P_dual)+1):end]])
  # Stack the Jacobian matrix and the matrix A
  m2x3 = vcat(jf, A)
  # Compute minors and saturate
  I = ideal(minors(m2x3, 2))
  J = ideal([jf[1, i] for i in 1:ncols(jf)])
  Isat = saturation(I + ideal([f_ext]), J)
  # Project to the dual space
  proj_dual_images = vcat([zero(P_dual) for _ in 1:nvars_P], gens(P_dual))
  proj = hom(P_ext, P_dual, proj_dual_images)
  dual_curve = groebner_basis(proj(Isat))
  return dual_curve[1]
end
\end{minted}
\noindent
We apply the function to a smooth plane quartic curve:
\begin{minted}{jlcon}
julia> P, (x, y, z) = graded_polynomial_ring(QQ, ["x", "y", "z"]);

julia> f = 8*x^4+20*x^2*y^2+8*y^4-48*x^2*z^2-48*y^2*z^2+65*z^4;

julia> P_dual, (u, v, w) = graded_polynomial_ring(QQ, ["u", "v", "w"]);

julia> f_dual = dual_curve(f, P_dual)
101920*u^12 - 283920*u^10*v^2 - 424704*u^10*w^2 - 329160*u^8*v^4 - 420192*u^8*v^2*w^2 + 701152*u^8*w^4 + 1211860*u^6*v^6 - 200976*u^6*v^4*w^2 + 1603016*u^6*v^2*w^4 - 585600*u^6*w^6 - 329160*u^4*v^8 - 200976*u^4*v^6*w^2 + 1873041*u^4*v^4*w^4 - 1405488*u^4*v^2*w^6 + 261000*u^4*w^8 - 283920*u^2*v^10 - 420192*u^2*v^8*w^2 + 1603016*u^2*v^6*w^4 - 1405488*u^2*v^4*w^6 + 489600*u^2*v^2*w^8 - 58752*u^2*w^10 + 101920*v^12 - 424704*v^10*w^2 + 701152*v^8*w^4 - 585600*v^6*w^6 + 261000*v^4*w^8 - 58752*v^2*w^10 + 5184*w^12
\end{minted}
\noindent
The curve $C$ and its dual are depicted in Figure~\ref{fig:dualcurve}.
 \end{example}

 \begin{figure}[ht]
\centering
\includegraphics[scale=0.12,angle=0]{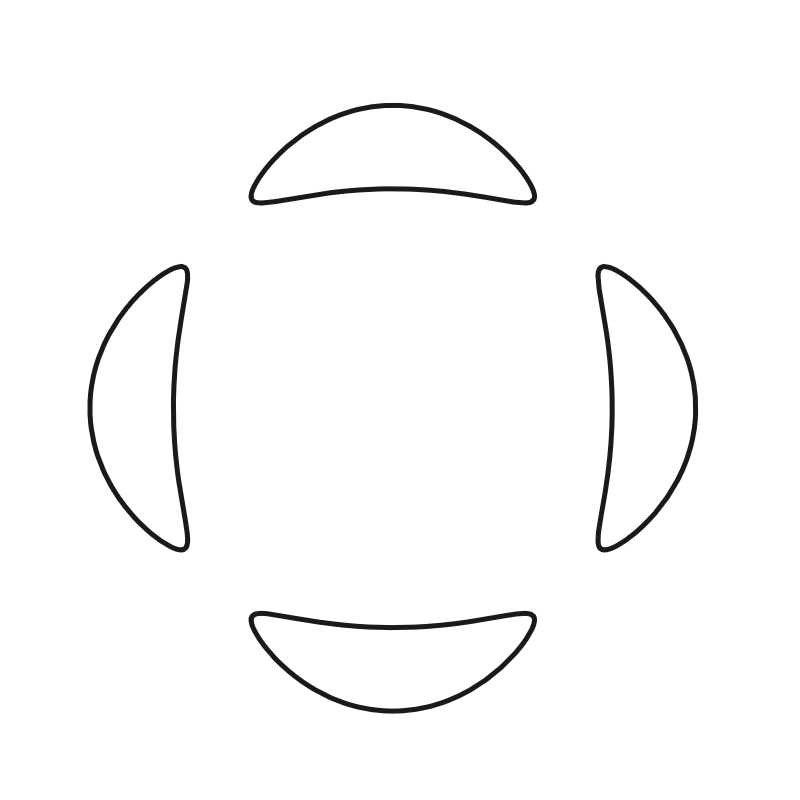}\quad
\includegraphics[scale=0.12,angle=0]{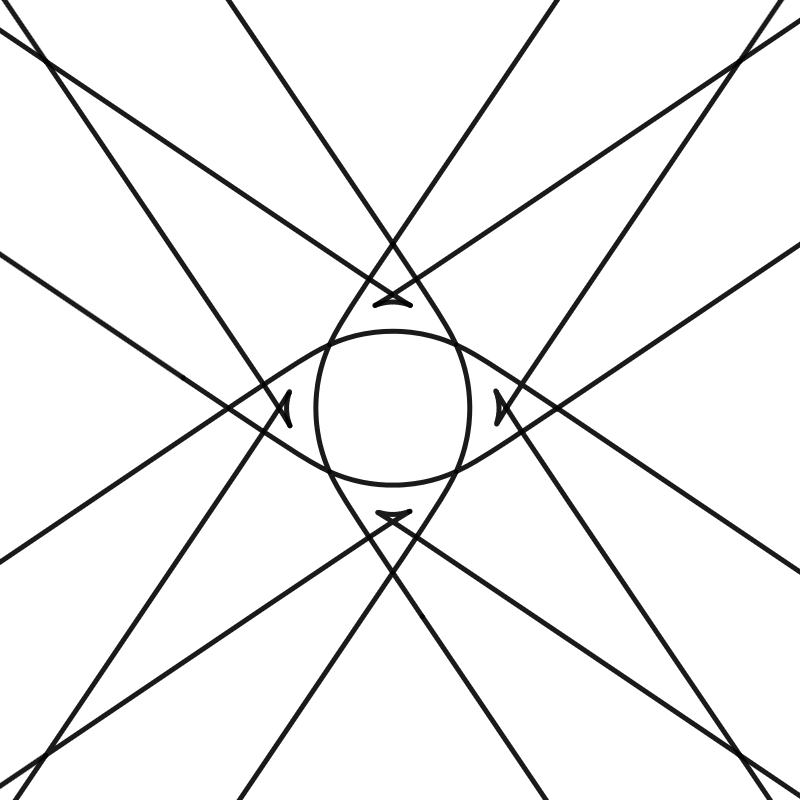}
\caption{Smooth plane quartic and its dual curve.}
\label{fig:dualcurve}
\end{figure}

By the Pl\"ucker formulas from the theorem below the curve $C$ in the example has $24$ flexes and $28$ bi-tangents. These correspond to $24$ ordinary cusps and $28$ ordinary double points of the dual curve.

\begin{theorem}[Pl\"ucker Formulas]\label{thm:Pluecker}
Let $C\subset \PP^2$ be an irreducible plane curve of degree $d$ defined over a field of characteristic $0$, and such that $C$
and the dual curve $\check C\subset \check \PP^2$ have only ordinary nodes and cusps as singularities. Let $f, b,\kappa, \delta$ denote the number of flexes, bi-tangents, cusps and nodes of $C$, respectively. Let $g$ denote the geometric genus of $C$, and let $\check d$ be the degree of $\check C$. Then
\begin{align*}
	g  &= \frac{1}{2}(d-1)(d-2)-\delta-\kappa = \frac{1}{2}(\check d-1),(\check d-2)-b-f, \\
	\check d&= d(d-1) - 2\delta-3\kappa,\quad\;  d= \check d(\check d-1) - 2b-3f,\\
   f & =3d(d-2)-6\delta-8\kappa, \quad \kappa =3\check d(\check d-2)-6b-8f,
\end{align*}
and
$b =\frac{1}{2} d(d-2)(d-3)(d+3)- (4d^2-4d-20)\delta-(6d^2-6d-27)\kappa+(2\delta + 3\kappa)^2.$
 \end{theorem}
 
 In Example \ref{ex:dual-curve}, eight of the flexes/cups are real and visible in Figure~\ref{fig:dualcurve}. All $28$ bi-tangents are real.    The $24$ flexes are precisely the $g^{3}-g$ Weierstrass points on this curve of genus $g=3$.

The only other possible gap sequence\footnote{The gap sequence at a point $p$ on a curve $C$ of genus $g$ is the sequence of $g$ integers that are not pole orders at $p$ of regular functions on $C\setminus \{p\}$. The point $p$ is a Weierstrass point if the sequence is different from $\{1,2,\ldots,g\}$.} of a Weierstrass point on a non-hyperelliptic curve of genus $3$ is $1,2,5$. In that case, the plane quartic has a tangent line which intersects the curve with multiplicity $4$ at the Weierstrass points and the dual curve becomes more singular than the singularities of curves allowed in the Pl\"ucker formulas.

Note that in the characteristic zero context of Theorem \ref{thm:Pluecker}, the double dual $\check{\check C}$ is identical to $C$. This explains the symmetry of the formulas. In positive 
characteristic, $\check{\check C}= C$ might not hold. 

What kind of singularities can occur in the dual curve of a smooth plane quartic
with a non-ordinary Weierstrass point?

\begin{example}\label{non-ordinary Weierstrass points}  Consider the quartic plane curves defined by
${\color{red} f_1=x^{4}+y^{3}-y}$
and ${\color{blue}f_2=x^{4}+y^{3}-\frac{2}{9}x^{2}-y+\frac{1}{81}}$. The curves are depicted in Figure~\ref{fig:curves2}.
\begin{figure}[ht]
\centering
\includegraphics[scale=0.15,angle=0]{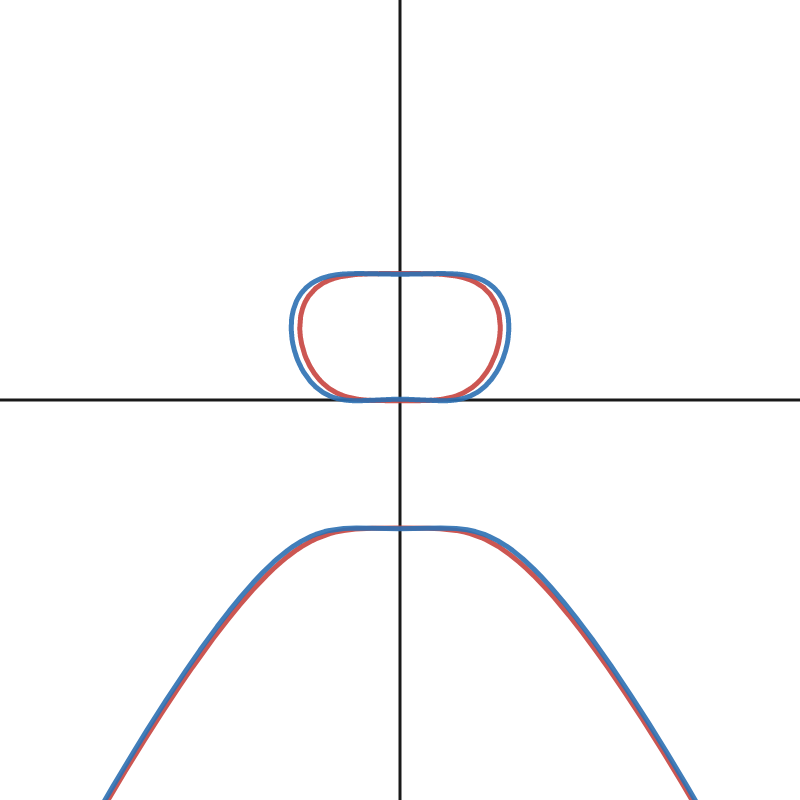}\quad
\includegraphics[scale=0.15,angle=0]{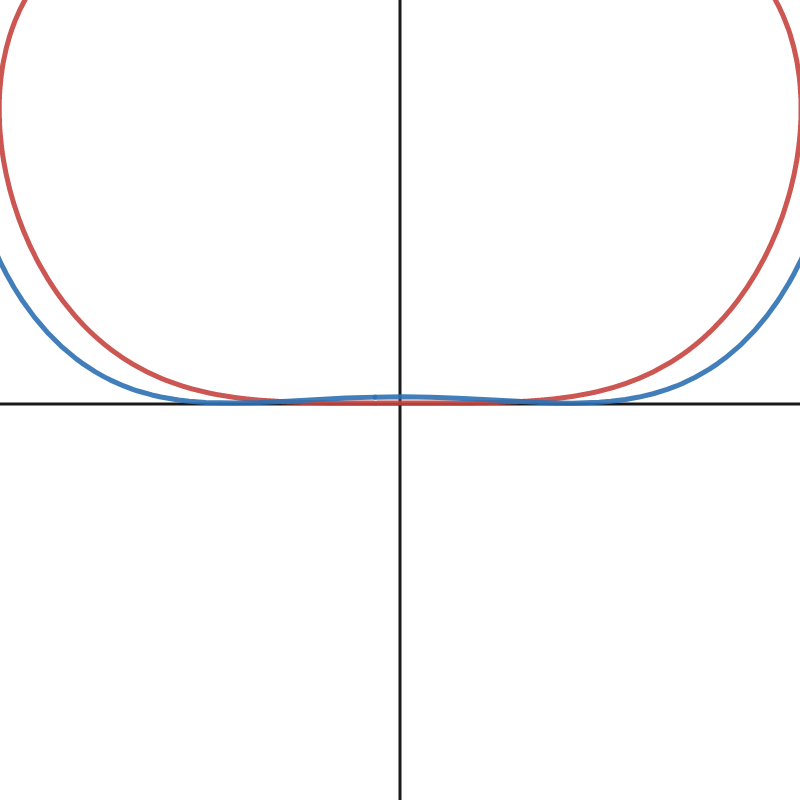}
\caption{Smooth quartic curves.}
\label{fig:curves2}
\end{figure} 
Their projective closures are smooth and the dual curves have the following affine equations:
\begin{minted}{jlcon}
julia> f1 = x^4+y^3*z-y*z^3;

julia> f1_dual = dual_curve(f1, P_dual);

julia> f1_dual(u,1,w)
4*u^12 - 48*u^8*w^3 + 48*u^8*w - 27*u^4*w^8 + 84*u^4*w^6 - 546*u^4*w^4 + 84*u^4*w^2 - 27*u^4 - 256*w^9 + 768*w^7 - 768*w^5 + 256*w^3

julia> f2 = x^4+y^3*z-y*z^3-2//9*x^2*z^2+1//81*z^4;

julia> f2_dual = dual_curve(f2, P_dual);

julia> f2_dual(u,1,w)
235953*u^12 + 8748*u^10*w^2 + 314928*u^10*w + 2916*u^10 - 118098*u^8*w^4 - 2834352*u^8*w^3 + 8748*u^8*w^2 + 2831760*u^8*w - 1458*u^8 + 708588*u^6*w^6 - 866052*u^6*w^4 + 90720*u^6*w^3 + 3700404*u^6*w^2 + 33696*u^6*w + 235940*u^6 - 1594323*u^4*w^8 + 4960116*u^4*w^6 - 1189728*u^4*w^5 - 32240754*u^4*w^4 + 46656*u^4*w^3 + 4967028*u^4*w^2 + 303264*u^4*w - 1589715*u^4 + 6928416*u^2*w^7 - 8188128*u^2*w^5 - 62208*u^2*w^4 + 3149280*u^2*w^3 - 1889568*u^2*w - 20736*u^2 - 15116544*w^9 + 45349632*w^7 + 186624*w^6 - 45349632*w^5 - 373248*w^4 + 15116544*w^3 + 186624*w^2
\end{minted}
\noindent
Figure \ref{fig:quartic_zoom} depicts the real points of the affine curves and a magnification of the plot around the point $p'=[0:1:0] \in \check \PP^2$.

\begin{figure}[ht]
  \centering
  \includegraphics[scale=0.15,angle=0]{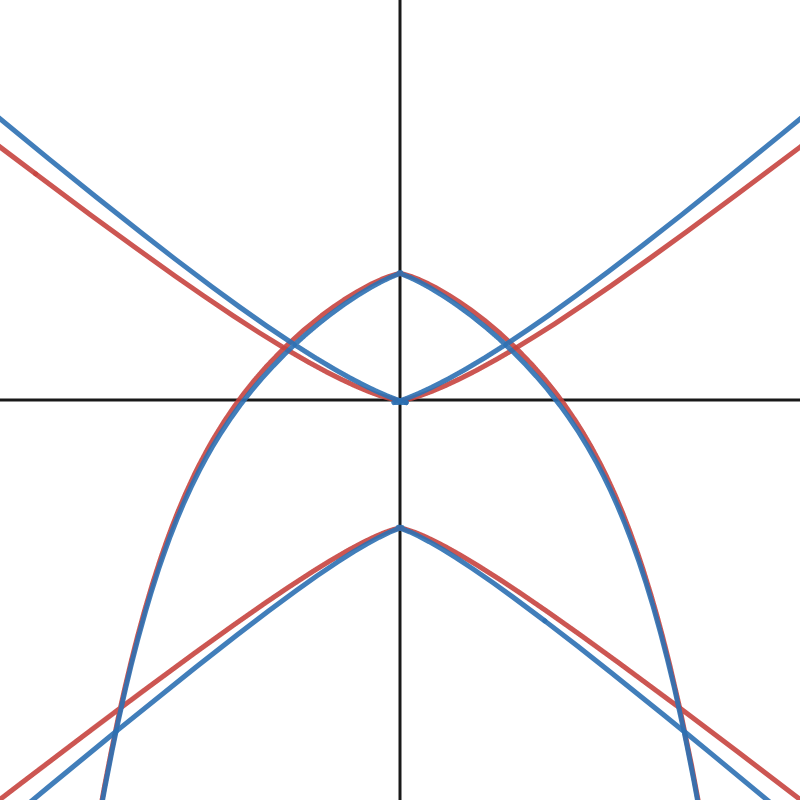}\quad 
  \includegraphics[scale=0.15,angle=0]{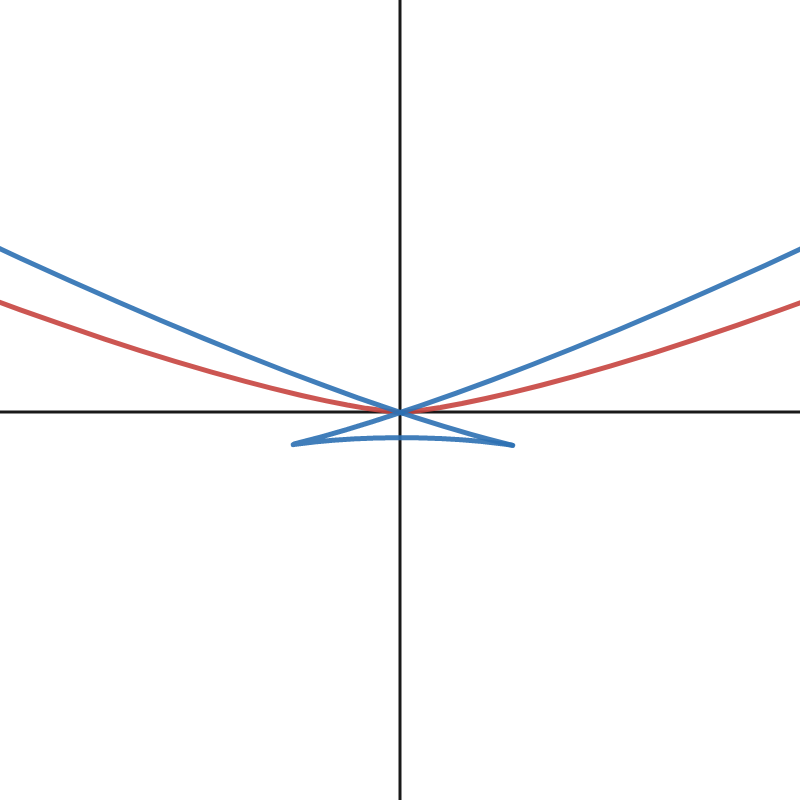}
  \caption{Plots of the real points of the curves and a magnification of the plot around the point $p'=[0:1:0] \in \check \PP^2$.}
  \label{fig:quartic_zoom}
\end{figure}
\noindent
The dual curve of $C_1=V(f_1)$ has an $E_6$-singularity at $p'$ which deforms to a node and two cusps in the dual curve of $C_2=V(f_2)$.
The corresponding point $p=[0:0:1]\in \PP^2$ is a Weierstrass point with gap sequence $1,2,5$.
Actually $C_1$ has $12$ non-ordinary Weierstrass points. There are no
other Weierstrass points since $2\cdot 12= 3^3-3$. Furthermore, $C_1$ has $16$ bi-tangents  apart from the tangents at the Weierstrass points, of which $4$ are real. 
\end{example}

\begin{example} It is known that for arbitrary genus $g$ a general curve
$C\in \cM_g$ has $g^3-g$ simple Weierstrass points and that the monodromy group over $\cM_g\setminus B$, where $B$ is the set of curves with a non-ordinary Weierstrass point or a non-trivial automorphism, 
is the full symmetric group
on $g^3-g$ points, see \cite{Eisenbud-Harris}. 

In case $g=3$ we can verify this fact via a Galois group computation.
Consider the family of plane curves defined by
$$V(s(8x^{4}+20x^{2}y^{2}+8y^{4}-48x^{2}z^2-48y^{2}z^2+65z^4)+t(x^3y+y^3z+z^3x)) \subset \PP^1\times \PP^2.$$
The set  of Weierstrass points in the fibers of this family is a curve $X$ 
with a $24:1$ map $X\to \PP^1$. One can verify that  the monodromy group of $X(\CC) \to \PP^1(\CC)$ is $S_{24}$ by  Galois group computations over $\QQ$ as follows. We consider two fibers, and commence showing that the corresponding two extension fields have Galois group $S_{24}$ over $\QQ$.
\begin{minted}{jlcon}
julia> P2, (x, y, z) = polynomial_ring(QQ,["x","y","z"]);

julia> f1 = 8*x^4+20*x^2*y^2+8*y^4-48*x^2*z^2-48*y^2*z^2+65*z^4+x*y^3;

julia> f2 = x^3*y+y^3*z+z^3*x+x^4;

julia> g1 = 1771*f1 - 1317*f2;

julia> hess1 = hessian(g1);

julia> I1 = ideal([hess1,g1]);

julia> g1yz = eliminate(I1,[x]);

julia> Q, t = polynomial_ring(QQ,"y");

julia> phi = hom(P2, Q, [0,t,1]);

julia> g1t = phi(g1yz[1]);

julia> G1,_ = galois_group(g1t);

julia> G1
Sym(24)

julia> g2 = 7713*f1 - 1313*f2;

julia> g2t = phi(eliminate(ideal([hessian(g2),g2]),[x])[1]);

julia> G2,_ = galois_group(g2t);

julia> G2==G1
true
\end{minted}
\noindent
Furthermore, we show that the extension fields do not have a common subfield other than $\QQ$: The only possible ramified primes are divisors of both polynomial discriminants. We observe that the gcd of the polynomial discriminants of the fibers is a power of 2.
\begin{minted}{jlcon}
julia> dg1 = discriminant(g1t);

julia> dg2 = discriminant(g2t);

julia> ggT = gcd(dg1,dg2);

julia> factor(ZZ(ggT))
1 * 2^4
\end{minted}
\noindent
Since the field discriminant (which contains exactly the ramified primes) 
divides the polynomial discriminant, 
only $2$ is a possible common prime divisor of the two field discriminants under consideration. Computing a maximal order over $2$ for one of the fibers shows 
that $2$ is not ramified in the normalization:
\begin{minted}{jlcon}
julia> KK, _ = number_field(g2t);

julia> degree(KK)
24

julia> O = any_order(KK);

julia> OO = pmaximal_overorder(O, 2);

julia> d = discriminant(OO);

julia> gcd(d,2)
1
\end{minted}
\noindent
Hence, $2$ cannot divide the 
discriminant of the intersection of the extension fields, which thus has discriminant one, and thus is $\QQ$.

Thus Galois group computations provide an alternative approach
to the commonly used
homotopy continuation technique.
\end{example}

Let $C\subset \PP^2$ be an irreducible plane curve of degree $d$ with only ordinary singularities $p_1,\ldots,p_s$ of multiplicities $r_1,\ldots, r_s$.
Using the completeness of the adjoint systems, we can compute for a divisor
$D$ on the normalisation $X=\widetilde C$ the Riemann--Roch space
$$H^0(X,\cO(D))=L(D)=\{f\in K(X)^* \mid (f)+D \ge 0\}\cup \{0\}$$ as follows.\\

\noindent
\textbf{Algorithm} (Riemann--Roch space) $\quad$  \\
\Input An irreducible plane curve $C'=V(f)$ with only ordinary singularities given by a square-free homogeneous polynomial $f \in K[x_0,x_1,x_2]$. \\
A divisor $D=\sum n_ip_i$  with support disjoint from the singularities of $C'$.  \\
\Output  $\ell(D)=\dim L(D)$ and a basis of $L(D)$. 

\begin{enumerate}
\item If the divisor is given by a list of pairs of multiplicities and points $\{ (n_i,p_i)\}$ compute
$$ \II(D_1) = \bigcap_{n_i>0} (\II(p_i)^{n_i}+(f)) \hbox{ and } \II(D_2)=\bigcap_{n_i<0} (\II(p_i)^{-n_i}+(f)). $$
\item Compute the adjoint ideal $J_{\textrm{adj}}$.
\item Verify that $C$ has only ordinary singularities, if necessary.
\item Compute $I = \II(D_1) \cap J_{\textrm{adj}}.$
\item Choose $e>0$ such that $I_e \supsetneq (f)_e$ and a form $h \in I_e \setminus (f)_e$.
\item Compute the residual ideal $I'= (f,h):I$. 
\item If $V(I'+\II(D_2))= \emptyset$ then $J = I' \cap \II(D_2)$ else $J=(I'  \cdot \II(D_2)+(f)):(x_0,x_1,x_2)^\infty.$
\item Compute $\ell=\ell(D)= \dim J_e/(f)_e$ and, if $\ell>0$, forms $h_1, \ldots h_\ell$, which represent a basis of $J_e/(f)_e$.
\item Return $\ell(D)$ and if $\ell(D)>0$ the  rational functions $h_1/h, \ldots, h_{\ell}/h$.
\end{enumerate}\vspace{1mm}

\begin{example}(Taken from \cite{Schreyer24}).
Consider the family of sextic curves
 $C_t= V(f_t) \subset \PP^2$
with affine parts defined by the polynomials
{\small
\begin{align*} f_t  =&xy(x+y)-3(x^5+y^5)-2(x^6+y^6)  \cr
&+t[-2xy+12(x^4+y^4)+x^3y+xy^3-20x^2y^2+8(x^5+y^5)-12(x^4y+xy^4) \cr
& +6(x^3y^2+x^2y^3)-5(x^5y+xy^5)-2(x^4y^2+x^2y^4)+14x^3y^3]\cr
&+t^2[-12(x^3+y^3)-2(x^2y+xy^2)-8(x^4+y^4)+24(x^3y+xy^3)\cr 
& -44x^2y^2+10(x^4y+xy^4)+20(x^3y^2+x^2y^3)+10(x^4y^2+x^2y^4)+24x^3y^3].
\end{align*}
}
\begin{figure}[t]
\centering
\includegraphics[scale=0.15]{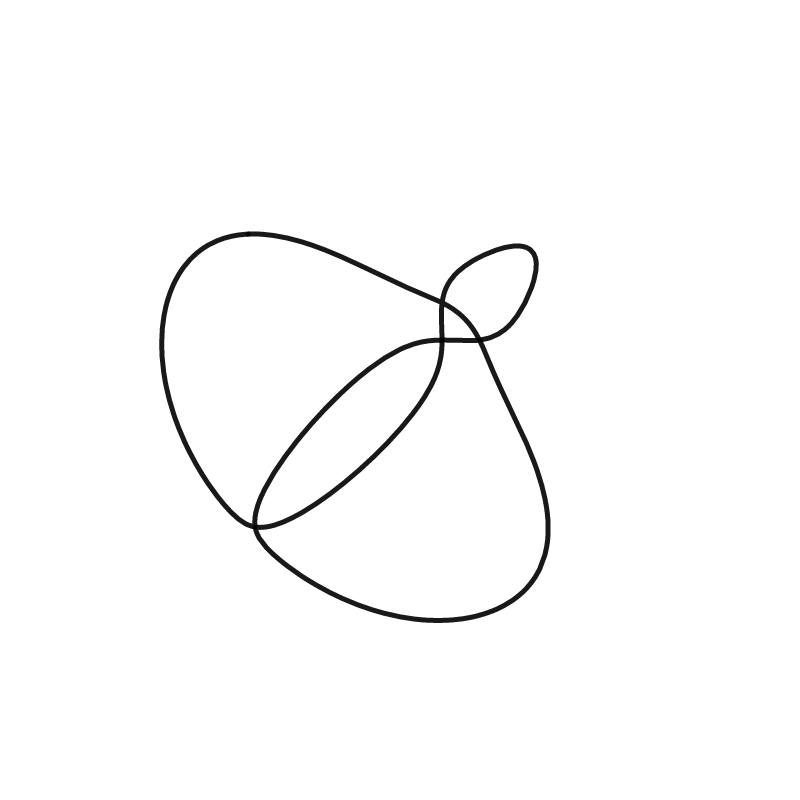}\quad
\includegraphics[scale=0.15]{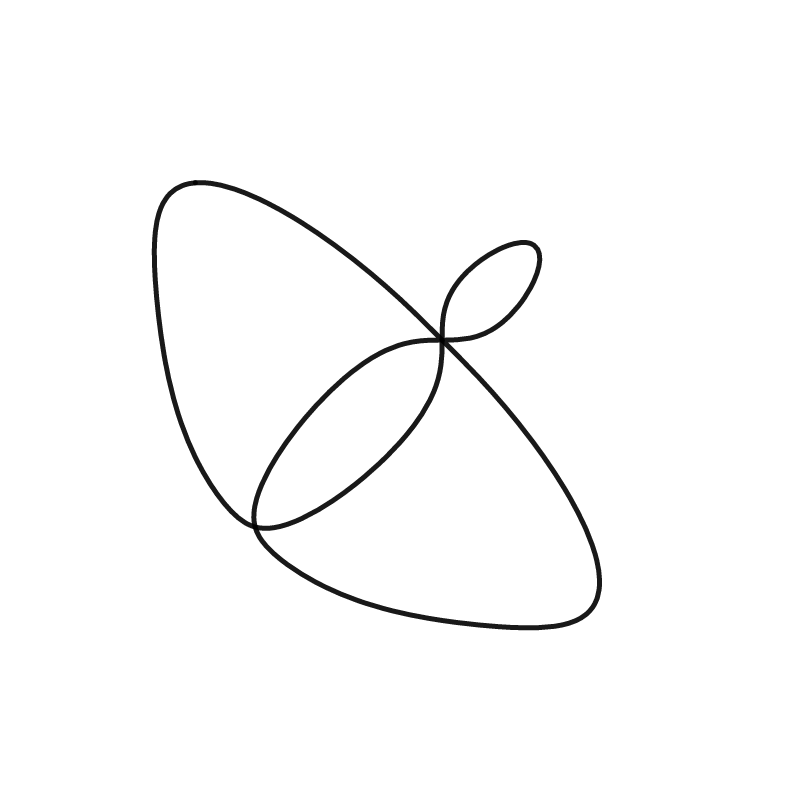}
\caption{Plots of the curves $C_{1/10}$ and $C_0$.}
\label{fig:genus6-configurations}
\end{figure}
\noindent
The curve $C_t$ has four ordinary double points
$p_0=(0,0),p_1=(2t,0),p_2=(0,2t)$, and $p_3=(-1,-1)$, 
and no further singularities for general values for $t$. 
However, $C_0$ has an ordinary triple point at $p_0$, see Figure~\ref{fig:genus6-configurations} for a plot of the curves $C_{1/10}$ and $C_0$. Thus this is a family of curves of genus $6$.

The curve $C_0$ is trigonal. The projection from the triple point induces a $3:1$ map $C\to \PP^1$.
From Petri's Theorem \cite{Petri1923}, the homogeneous ideal of the canonical model of this curve in $\PP^5$
needs cubic generators. We check this computationally:
 
\begin{minted}{jlcon}
julia> Rt, (x,y,t) = polynomial_ring(QQ, ["x","y","t"]);

julia> ft = -2*x^6 - 5*x^5*y*t + 8*x^5*t - 3*x^5 + 10*x^4*y^2*t^2 - 2*x^4*y^2*t + 10*x^4*y*t^2 - 12*x^4*y*t - 8*x^4*t^2 + 12*x^4*t + 24*x^3*y^3*t^2 + 14*x^3*y^3*t + 20*x^3*y^2*t^2 + 6*x^3*y^2*t + 24*x^3*y*t^2 + x^3*y*t - 12*x^3*t^2 + 10*x^2*y^4*t^2 - 2*x^2*y^4*t + 20*x^2*y^3*t^2 + 6*x^2*y^3*t - 44*x^2*y^2*t^2 - 20*x^2*y^2*t - 2*x^2*y*t^2 + x^2*y - 5*x*y^5*t + 10*x*y^4*t^2 - 12*x*y^4*t + 24*x*y^3*t^2 + x*y^3*t - 2*x*y^2*t^2 + x*y^2 - 2*x*y*t - 2*y^6 + 8*y^5*t - 3*y^5 - 8*y^4*t^2 + 12*y^4*t - 12*y^3*t^2;

julia> R, (x,y) = polynomial_ring(QQ, ["x","y"]);

julia> phi0 = hom(Rt,R,[x,y,0]);

julia> f = phi0(ft);

julia> H = homogenizer(R, "z");

julia> F = H(f);

julia> S = parent(F);

julia> C = plane_curve(F);

julia> I = adjoint_ideal(C)
Ideal generated by
  x*y - y^2
  x^2 - y^2
  y^3 + y^2*z

julia> m3 = ideal(gens(S))^3;

julia> K = intersect(I, m3)
Ideal generated by
  x*y*z - y^2*z
  x^2*z - y^2*z
  y^3 + y^2*z
  x*y^2 + y^2*z
  x^2*y + y^2*z
  x^3 + y^2*z

julia> P5, _ = graded_polynomial_ring(QQ, ["x_0", "x_1", "x_2", "x_3", "x_4", "x_5"]);

julia> PC, pr = quo(S, ideal([F]));

julia> psi = hom(P5, PC, [pr(K[i]) for i = 1:6]);

julia> J = kernel(psi);

julia> JJ = ideal(minimal_generating_set(J));

julia> Q , _ = quo(P5, JJ);

julia> re = free_resolution(Q);

julia> minimal_betti_table(re)
       0  1   2  3  4
---------------------
0    : 1  -   -  -  -
1    : -  6   8  3  -
2    : -  3   8  6  -
3    : -  -   -  -  1
---------------------
total: 1  9  16  9  1
\end{minted}
\noindent
The curves $C_t$ for $t\not=0$, on the other hand, are 4-gonal
and their canonical images are generated by quadrics alone.
We verify this for the curve $C_{1/10}$.

\begin{minted}{jlcon}

julia> f = phi1(ft);

julia> F = H(f);

julia> C = plane_curve(F);

julia> I = adjoint_ideal(C)
Ideal generated by
  6*x*y - 5*y^2 + y*z
  5*x^2 - x*z - 5*y^2 + y*z

julia> m3 = ideal(gens(S))^3;

julia> K = intersect(I, m3)
Ideal generated by
  6*x*y*z - 5*y^2*z + y*z^2
  5*x^2*z - x*z^2 - 5*y^2*z + y*z^2
  5*y^3 + 4*y^2*z - y*z^2
  6*x*y^2 + 5*y^2*z - y*z^2
  6*x^2*y + 5*y^2*z - y*z^2
  25*x^3 - x*z^2 + 20*y^2*z - 4*y*z^2

julia> PC, pr = quo(S, ideal([F]));

julia> psi = hom(P5, PC, [pr(K[i]) for i = 1:6]);

julia> J = kernel(psi);

julia> JJ = ideal(minimal_generating_set(J));

julia> Q , _ = quo(P5, JJ);

julia> re = free_resolution(Q);

julia> minimal_betti_table(re)
       0  1   2  3  4
---------------------
0    : 1  -   -  -  -
1    : -  6   5  -  -
2    : -  -   5  6  -
3    : -  -   -  -  1
---------------------
total: 1  6  10  6  1

\end{minted}
\noindent

\end{example}

\begin{remark} The phenomenon above has a wide generalization.
According to Green's Conjecture \cite{Green84} for smooth projective curves defined over $\CC$ one can read off the gonality (more precisely the Clifford index) of the curve from the Betti table of its canonical model. This was proved in landmark papers by Voisin \cite{Voisin02, Voisin05} for the general curve using $K3$ surfaces and in \cite{AF11} for curves on arbitrary $K3$ surfaces. The conjecture is known to be wrong for some smooth curves over some fields of positive characteristic
\cite{Schreyer86,BoppSchreyer21}. Green's conjecture is open in its full generality.    
\end{remark}

\section{Deformations}\label{Deformations}

\begin{example}
The surface from Example \ref{ratd=11pi=10} is defined over the field $\FF_3$. Actually, we would like to prove the existence of an example over $\CC$.
The key point in the construction is the module $S/m$ with Hilbert function
$(1,5,5,0,\ldots)$. A general module of this Hilbert function
has Betti numbers
 \[
\begin{tightarray}{l c r r r r r r}
    & & 0 & 1 & 2 & 3 & 4 & 5 \\
\hline
0 &:   & 1 & - & - & - & - & - \\
1  &: & - & 10 & 15 & - & - & - \\
2  &: & - & - & 5 & 26 & 20 & 5 \\
\hline
\text{total}&: & 1 & 10 & 20 & 26 & 20 & 5 \\
\end{tightarray}
\quad \hbox{instead of} \quad
\begin{tightarray}{l c r r r r r r}
    & & 0 & 1 & 2 & 3 & 4 & 5 \\
\hline
0 &:   & 1 & - & - & - & - & - \\
1  &: & - & 10 & 15 & 2 & - & - \\
2  &: & - & - & 7 & 26 & 20 & 5 \\
\hline
\text{total}&: & 1 & 10 & 22 & 28 & 20 & 5 \\
\end{tightarray}
\]

The family of modules
$$M = \{ \hbox{good } m's\} \subset \GG(10,H^0(\PP^4,\cO(2))$$
has at the given point $m$ from Example \ref{ratd=11pi=10} at most codimension $7\cdot 2=14$. We will show by computation that equality holds
and that $M$ is smooth of dimension $10\cdot 5 -14=36$ at our given $m$. We start out by determining the free resolution:
\begin{minted}{jlcon}
julia> K = GF(3);

julia> S, (x0, x1, x2, x3, x4) = graded_polynomial_ring(K, ["x0", "x1", "x2", "x3", "x4"]);

julia> m = ideal(S, [x1^2+(-x1+x2+x3-x4)*x0, x1*x2+(x1-x3+x4)*x0, x1*x3+(-x1+x4+x0)*x0, x1*x4+(-x1+x3+x4-x0)*x0, x2^2+(x1-x2-x4-x0)*x0, x2*x3+(x1-x2+x3+x4-x0)*x0, x2*x4+(x1+x2-x3-x4-x0)*x0, x3^2+(x3+x4-x0)*x0,x3*x4+(-x3-x4+x0)*x0, x4^2+(x1+x3-x4-x0)*x0]);

julia> R, _ = quo(S, m);

julia> FR = free_resolution(R, algorithm = :mres);
\end{minted}
\noindent
Let $\varphi_1,\ldots,\varphi_3$ denote the first three maps in the resolution of $S/m$. We add a $1\times 10$ deformation matrix $A_1$ to $\varphi_1$
and solve recursively the equations
$$ A_1\cdot \varphi_2=-\varphi_1\cdot A_2,\quad A_2\cdot \varphi_3=-\varphi_2\cdot A_3.$$
We write a helper function to solve this problem, finding the normal space generator-by-generator:

\begin{minted}{jlcon}
julia> L = monomial_basis(R, 2)
5-element Vector{MPolyDecRingElem{FqFieldElem, FqMPolyRingElem}}:
 x4^2
 x3*x4
 x2*x4
 x2*x3
 x1*x4

julia> versal_unfolding = [[i == div((j-1), 5) + 1 ? S(L[(j-1) % 5 + 1]) : S(0) for i in 1:10] for j in 1:50];

julia> function normal_space_generator(FR, A1t)
         phi1 = map(FR, 1)
         e1 = gen(codomain(phi1),1)
         phi2 = map(FR, 2)
         phi3 = map(FR, 3)
         A1 = hom(domain(phi1), codomain(phi1), [p*e1 for p in A1t])
         A1phi2 = phi2 * A1
         A2 = lift(A1phi2, phi1)
         A2phi3 = phi3 * A2
         A3 = lift(A2phi3, phi2)
         A3m = matrix(A3)
         return A3m[1:2, 16:22]
       end;

julia> nlist = [normal_space_generator(FR, v) for v in versal_unfolding];

julia> TS, t = graded_polynomial_ring(K, "t"=>(1:50));

julia> nlist_t = [map_entries(x -> x * t[i], map_entries(constant_coefficient, nlist[i])) for i in 1:50];

julia> B = sum(nlist_t);
\end{minted}
\noindent
The complex with differentials
$$\varphi_1+A_1, \varphi_2+A_2 \hbox{ and } \varphi_3+A_3 $$
is a flat first order deformation defined over $\FF_3[t_1,\ldots,t_{50}]/(t_1,\ldots,t_{50})^2$.
The entries of the $7\times 2$ submatrix $B$ of $A_3$ returned by the function
define the tangent space 
$$T_m M \subset T_m\GG(10,H^0(\PP^4,\cO(2))=\Hom(\FF_3^{10},\FF_3^5).$$
\begin{minted}{jlcon}
julia> tangent_space = ideal(vec(collect(B)));

julia> ngens(leading_ideal(tangent_space))
14
\end{minted}
\noindent
Now, $M$ is defined over the integers. Taking linear equations $\ell_1,\ldots,\ell_{14}$ on $\PP(\Lambda^5 H^0(\PP^4,\cO(2))^*)$
defined over $\ZZ$ which cut mod $3$ the variety $M(\FF_3)$ transversal at $m$, we get a scheme $V(\ell_1,\ldots,\ell_{14})\cap M$ over $\Spec \ZZ$
which has as one of its components an open part of $\Spec \sO_L$
of a number field $L$ and a prime ideal $\pp$ over $(3)$ with $\sO_{L,\pp}/\pp\cong \FF_3$.
The other steps in the construction of the surface work over
$\Spec \sO_{L,\pp}$, since the conditions that the construction works are open conditions. In particular, the generic fiber of the family is a smooth surface defined over $L$ with the same adjunction behavior as the surface over the finite field.
\end{example}

\begin{remark}
Notice that this technique allows one to prove the existence of interesting objects in algebraic geometry with a search using computer algebra.
The example above was found by randomly picking points $p\in \GG(10,15)(\FF_3)$
and testing for $2$ extra syzygies. The probability for a random $p$ to lie in $M(\FF_3)$ is roughly $1:3^{14}$ because the codimension $M$ in $\GG(10,15)$ is $14$. We can improve our search by searching in the unirational subvariety
$$M_1=\{ m \in \GG(10,H^0(\PP^4,\cO(2))) \mid \dim \Tor_3^S(S/m,\FF_3)_4 \ge 1 \}$$
in which $M$ has codimension at most $7$. Four families of smooth non-general type surfaces with $p_g=q=0$, degree $d=11$, and sectional genus $\pi=10$
were found with this method \cite{zbMATH00943474}. This is interesting because of the afore\-mentioned result of Ellingsrud and Peskine \cite{EP89} which states that there are only finitely 
many components of the Hilbert scheme of  $\PP^4$ whose general points correspond to smooth surfaces of Kodaira dimension $<2$, confirming a conjecture 
of Hartshorne about smooth rational surfaces in $\PP^4$. Currently there are $24$ families of smooth rational surfaces known in $\PP^4$. The largest degree of a known 
smooth rational surface in $\PP^4$ is $12$, see \cite{AboRanestad}.
\end{remark}
    
A problem on which the cornerstone system \Singular grew is the search by Greuel and Pfister for a counterexample to the Zariski Conjecture \cite{greuel2021history}:

\textit{Is the multiplicity of an isolated singularity of a hypersurface $V(f) \subset \CC^n$ a topological invariant?}\\
Let $f \in \CC\{x,y,z\}$ be a convergent power series with an isolated singularity at $0$. Recall that the \emph{Milnor fiber}
$$M_f= V(f-t) \cap S^5_\varepsilon \subset \CC^3$$
is homotopic to a bouquet of $\mu$ spheres $S^2$, where $\mu$ is the \emph{Milnor number}
$$\mu= \dim_\CC \CC\{x,y,z\}/(\frac{\partial f}{\partial x},\frac{\partial f}{\partial y},\frac{\partial f}{\partial z}).$$
The \emph{Tjurina number}
$$\tau=\dim_\CC \CC\{x,y,z\}/(f,\frac{\partial f}{\partial x},\frac{\partial f}{\partial y},\frac{\partial f}{\partial z})$$
is the dimension of the (smooth) base space $B$ of the semi-universal deformation of $f$. Zariski asked whether the multiplicity is constant along the $\mu$-constant stratum $\subset B$.  

The speed with which \Singular operates comes from the desire of Greuel and Pfister to treat examples
with five digit Milnor numbers.

\begin{example} Consider 
$$f_t=x^a+y^b + z^{3c} + x^{c+2}y^{c-1} + x^{c-1}y^{c-1}z^3 + x^{c-2}y^c(y^2 + tx)^2$$
with $(a,b,c)=(40,30,8).$ In \OSCAR, multiplicities  and Milnor numbers can be found in the spirit
of Section \ref{sect:local-studies} using local monomial orderings. We get:
$$m(f_0) = 17,\, m(f_t) = 16,\, \mu(f_0) = 10661,\, \mu(f_t) = 10655.$$
For example\footnote{In experiments which aim at finding a particular example in characteristic zero, it is often 
convenient to first work over a finite field to speed up computations. Once a promising example has been found, and no
theoretical argument is known for deducing the existence of an example in characteristic zero, one has to verify the computation
in characteristic zero. In the example under consideration here, the latter is feasible due an algorithm given in \cite{GPS2022}.}:
\begin{minted}{jlcon}
julia> k = GF(31991);

julia> S, (x, y, z, t) = polynomial_ring(k, ["x", "y", "z", "t"]);

julia> a, b, c = 40, 30, 8
(40, 30, 8)

julia> ft=x^a+y^b+z^(3*c)+x^(c+2)*y^(c-1)+x^(c-1)*y^(c-1)*z^3+x^(c-2)*y^c*(y^2+t*x)^2;

julia> R, (x, y, z) = polynomial_ring(k, ["x", "y", "z"]);

julia> f0 = hom(S, R, [x, y, z, 0])(ft);

julia> f1 = hom(S, R, [x, y, z, 1])(ft);

julia> MI0 = jacobian_ideal(f0);    MI1 = jacobian_ideal(f1);

julia> U = complement_of_point_ideal(R, [0 ,0, 0]);

julia> Rloc, phi = localization(R, U);

julia> A0, _ = quo(Rloc, phi(MI0));  A1, _ = quo(Rloc, phi(MI1));

julia> vector_space_dimension(A0)
10661

julia> vector_space_dimension(A1)
10655
\end{minted}
\noindent
The multiplicity goes down for $t\not=0$. Unfortunately (or luckily), the Milnor number goes down as well. Thus $f_t$ is not a
counterexample to the Zariski Conjecture. However, $\frac{6}{10661}\approx  .00056$ is rather small compared to $\frac{1}{17}\approx .059$. 

\end{example}

\section*{Acknowledgements}
Gefördert durch die Deutsche Forschungsgemeinschaft (DFG) - Projektnummer 286237555 - TRR 195 [Funded by the Deutsche Forschungsgemeinschaft (DFG, German Research Foundation) - Project- ID 286237555 - TRR 195]. The work of JB and WD was supported by Project B5 of SFB-TRR 195. The work of FS was supported by Project A23 of SFB-TRR 195. JB and WD also acknowledge support of the Potentialbereich
\emph{SymbTools – Symbolic Tools in Mathematics and their Application} of the Forschungsinitiative Rheinland-Pfalz.

\printbibliography

\end{document}